\documentclass{daj}

\dajAUTHORdetails{%
  title = {Uniform Finite Presentation for Groups of Polynomial Growth}, 
  author = {Philip Easo and Tom Hutchcroft},
  plaintextauthor = {Philip Easo, Tom Hutchcroft},
}   

\dajEDITORdetails{%
   year={2025},
   number={1},
   received={20 September 2023},   
   published={30 April 2025},  
   doi={10.19086/da.127778},       
}   

\usepackage{amsmath,amssymb,amsfonts,amsthm}
\usepackage{color}

\usepackage{cleveref}

\usepackage{caption}
\usepackage{thmtools}

\usepackage{mathrsfs}

\crefname{theorem}{Theorem}{Theorems}
\crefname{thm}{Theorem}{Theorems}
\crefname{lemma}{Lemma}{Lemmas}
\crefname{claim}{Claim}{Claims}
\crefname{lem}{Lemma}{Lemmas}
\crefname{remark}{Remark}{Remarks}
\crefname{prop}{Proposition}{Propositions}
\crefname{proposition}{Proposition}{Propositions}
\crefname{defn}{Definition}{Definitions}
\crefname{definition}{Definition}{Definitions}
\crefname{corollary}{Corollary}{Corollaries}
\crefname{conjecture}{Conjecture}{Conjectures}
\crefname{question}{Question}{Questions}
\crefname{chapter}{Chapter}{Chapters}
\crefname{section}{Section}{Sections}
\crefname{part}{Part}{Parts}
\crefname{figure}{Figure}{Figures}

\theoremstyle{plain}
\newtheorem{thm}{Theorem}[section]

\newtheorem{lemma}[thm]{Lemma}

\newtheorem{lem}[thm]{Lemma}

\newtheorem{corollary}[thm]{Corollary}

\newtheorem{prop}[thm]{Proposition}
\newtheorem{proposition}[thm]{Proposition}

\theoremstyle{definition}

\newtheorem{example}[thm]{Example}

\theoremstyle{remark}
\newtheorem{remark}{Remark}[section]
\newtheorem*{remark*}{Remark}

\numberwithin{equation}{section}

\newcommand{\R}{\mathbb R}
\newcommand{\Z}{\mathbb Z}

\usepackage{extarrows}
\usepackage{commath}
\usepackage{mathtools}
\newcommand{\eps}{\varepsilon}
\usepackage{bbm}
\usepackage{setspace}
\usepackage{enumitem}
\usepackage{cite}

\usepackage{tikz}

\newcommand{\Vol}{\operatorname{vol}}

\numberwithin{equation}{section}

\usepackage{mathrsfs}

\newcommand{\Cay}{\operatorname{Cay}}

\newcommand{\Gr}{\operatorname{Gr}}

\newcommand{\spanZ}{\operatorname{span}_\Z}

\makeatletter
\newcommand{\xRrightarrow}[2][]{\ext@arrow 0359\Rrightarrowfill@{#1}{#2}}
\newcommand{\Rrightarrowfill@}{\arrowfill@\equiv\equiv\Rrightarrow}
\newcommand{\xLleftarrow}[2][]{\ext@arrow 3095\Lleftarrowfill@{#1}{#2}}
\newcommand{\Lleftarrowfill@}{\arrowfill@\Lleftarrow\equiv\equiv}
\newcommand{\xLleftRrightarrow}[2][]{\ext@arrow 3399\LleftRrightarrowfill@{#1}{#2}}
\newcommand{\LleftRrightarrowfill@}{\arrowfill@\Lleftarrow\equiv\Rrightarrow}
\makeatother

\DeclarePairedDelimiter\pnorm{(\!(}{)\!)}

\makeatletter
\DeclareFontFamily{OMX}{MnSymbolE}{}
\DeclareSymbolFont{MnLargeSymbols}{OMX}{MnSymbolE}{m}{n}
\SetSymbolFont{MnLargeSymbols}{bold}{OMX}{MnSymbolE}{b}{n}
\DeclareFontShape{OMX}{MnSymbolE}{m}{n}{
    <-6>  MnSymbolE5
   <6-7>  MnSymbolE6
   <7-8>  MnSymbolE7
   <8-9>  MnSymbolE8
   <9-10> MnSymbolE9
  <10-12> MnSymbolE10
  <12->   MnSymbolE12
}{}
\DeclareFontShape{OMX}{MnSymbolE}{b}{n}{
    <-6>  MnSymbolE-Bold5
   <6-7>  MnSymbolE-Bold6
   <7-8>  MnSymbolE-Bold7
   <8-9>  MnSymbolE-Bold8
   <9-10> MnSymbolE-Bold9
  <10-12> MnSymbolE-Bold10
  <12->   MnSymbolE-Bold12
}{}

\let\llangle\@undefined
\let\rrangle\@undefined
\DeclareMathDelimiter{\llangle}{\mathopen}%
                     {MnLargeSymbols}{'164}{MnLargeSymbols}{'164}
\DeclareMathDelimiter{\rrangle}{\mathclose}%
                     {MnLargeSymbols}{'171}{MnLargeSymbols}{'171}
\makeatother

\begin{document}

\begin{frontmatter}[classification=text]


\author[philip]{Philip Easo}
\author[tom]{Tom Hutchcroft\thanks{Supported by NSF grant
DMS-2246494.}}

\begin{abstract}
We prove a quantitative refinement of the statement that groups of polynomial growth are finitely presented.
Let $G$ be a group with finite generating set $S$ and let $\operatorname{Gr}(r)$ be the volume of the ball of radius $r$ in the associated Cayley graph. For each $k \geq 0$, let $R_k$ be the set of words of length at most $2^k$ in the free group $F_S$ that are equal to the identity in $G$, and let $\llangle R_k \rrangle$ be the normal subgroup of $F_S$ generated by $R_k$, so that the quotient map $F_S/\llangle R_k\rrangle \to G$ induces a covering map of the associated Cayley graphs that has injectivity radius at least $2^{k-1}-1$. Given a non-negative integer $k$, we say that $(G,S)$ has a \textbf{new relation on scale k} if $\llangle R_{k+1} \rrangle \neq \llangle R_{k} \rrangle$. We prove that for each $K<\infty$ there exist constants $n_0$ and $C$ depending only on $K$ and $|S|$ such that if $\Gr(3n)\leq K \Gr(n)$ for some $n\geq n_0$, then there exist at most $C$ scales $k\geq \log_2 (n)$ on which $G$ has a new relation. We apply this result in another paper as part of our proof of Schramm's locality conjecture in percolation theory.
\end{abstract}
\end{frontmatter}


\section{Introduction}

It is a seminal theorem of Gromov \cite{gromov81poly} (see also \cite{MR2629989,ozawa2018functional}) that a finitely generated group has polynomial volume growth if and only if it is virtually nilpotent. This theorem and its extension to transitive graphs due to Trofimov \cite{MR811571} are of foundational importance in the study of geometry and probability on transitive graphs, implying in particular that every transitive graph of polynomial growth has a well-defined volume growth dimension and that this dimension is an integer. In probability, these theorems are often used together with the isoperimetric inequality of Coulhon and Saloff-Coste \cite{MR1232845} to prove results for general transitive graphs via a ``structure vs.\ expansion'' dichotomy: that is, proceeding by a case analysis according to whether the graph is virtually nilpotent or satisfies a $d$-dimensional isoperimetric inequality for every $d<\infty$. Important results in probability employing the structure theory of transitive graphs in this way include Varopoulos's theorem \cite{MR832044} that an infinite transitive graph is recurrent for simple random walks if and only if it has linear or quadratic volume growth, and Duminil-Copin, Goswami, Severo, Raoufi, and Yadin's proof that transitive graphs admit a percolation phase transition if and only if they have superlinear growth \cite{1806.07733}.

\medskip

Over the last twenty years, an extensive literature in \emph{approximate group theory} has been developed establishing \emph{finitary} versions of Gromov's theorem and Trofimov's theorem, highlights of which include \cite{bgt12,shalom-tao,MR2833482,tt.Trof,MR2827010}. See \cite{breuillard2014brief} for a detailed overview, \cite{MR3971253} for a textbook introduction, and \cite{tointon2020brief} for a concise survey.
For groups, this theory culminated in the celebrated work of Breuillard, Green, and Tao \cite{bgt12}, a special case of whose results can be stated\footnote{We state their theorem in a `metric' form that is convenient for our applications, and which is adapted from Tessera and Tointon's structure theorem for \emph{vertex-transitive graphs} of polynomial growth \cite[Theorem 2.3]{tt.Trof}. Indeed, the statement given below is equivalent to the special case of their theorem in which the graph $\Gamma$ is the Cayley graph of $G$, together with the growth bound of \cite[Corollary 11.9]{bgt12}. 
} as follows. Given a group $G$ and a finite generating set $S$, we write $\Gr(r)=\Gr_{G,S}(r)$ for the cardinality of the ball of radius $r$ in the Cayley graph $\Cay(G,S)$.

\begin{thm}[Breuillard, Green, and Tao 2012]
\label{thm:BGT_metric}
For each $K\geq 1$ there exist constants $r_0=r_0(K)$ and $C=C(K)$ such that the following holds. Let $G$ be a group with finite generating set $S$, and suppose that there exists $r \geq r_0$ such that $\Gr(3r) \leq K\Gr(r)$. Then $\Gr(mr) \leq m^C \Gr(r)$ for every $m\geq 3$ and there exists a finite normal subgroup $Q \triangleleft G$ such that:

\begin{enumerate}
    \item Every fibre of the projection $\pi:G \rightarrow G/Q$ has diameter at most $Cr$.
    
    
    \item $G/Q$ has a nilpotent normal subgroup $N$ of rank, step and index at most $C$.
    
    
    
    \item The projection $x\mapsto \pi(x)$ is a $(1,Cr)$-quasi-isometry from $\operatorname{Cay}(G,S)$ to $\operatorname{Cay}(G/Q,\pi(S))$.
\end{enumerate}
\end{thm}

Here, we recall that a function $\phi:V_1\to V_2$ between the vertex sets of two graphs $G_1=(V_1,E_1)$ and $G_2=(V_2,E_2)$ is said to be an $(\alpha,\beta)$\textbf{-quasi-isometry} (a.k.a.\ \textbf{rough isometry}) if 
\[\alpha^{-1} d(x,y) -\beta \leq d(\phi(x),\phi(y)) \leq \alpha d(x,y) +\beta\]
for every $x,y\in V_1$, and every vertex $z\in V_2$ is within distance at most $\beta$ of $\phi(V_1)$. (The second property holds automatically if $\phi$ is surjective.)

\medskip

Informally, the Breuillard-Green-Tao theorem states that polynomial growth at one sufficiently large scale forces the group to have polynomial growth at every subsequent scale, and moreover to be metrically ``well-modelled'' by a nilpotent group at all larger scales. Similar theorems for vertex-transitive graphs that are not necessarily Cayley graphs have recently been established in the work of Tessera and Tointon \cite{tt.Trof,MR3877012}.

\medskip

 These results have recently found many probabilistic applications, particularly for problems concerning \emph{families} of transitive graphs (such as sequences of finite transitive graphs converging to an infinite graph); such problems often require estimates that are ``uniform in the graph'', so that  structure theoretic results invoked in their solutions must typically be finitary. Results proven using finitary structure theory include a finite-graph version of Varopoulos's theorem \cite{tt.resist}, universality theorems for cover time fluctuations \cite{berestycki2022universality}, locality of the critical probability for graphs of polynomial growth \cite{contreras2023locality}, non-triviality of the supercritical phase for percolation on finite transitive graphs \cite{hutchcroft2021non}, and ``gap at $1$'' theorems for the critical probability on infinite vertex transitive graphs \cite{hutchcroft2021non,lyons2023explicit,panagiotis2021gap}. Several of these works exploit finitary versions of the ``structure vs.\ expansion'' dichotomy provided by the finitary structure theory of \cite{bgt12,tt.Trof}, with key technical difficulties arising from the fact that the same graph may exhibit different sides of this dichotomy at different scales.

\subsection{Uniform finite presentation}

Since virtually nilpotent groups are finitely presented, it is a consequence of Gromov's theorem that every group of polynomial volume growth is finitely presented. The purpose of this paper is to prove a \emph{uniform} version of this fact, stating roughly that every group of polynomial growth has a bounded number of scales witnessing a new relation after the first scale that polynomial growth is witnessed. This result is used in our work \cite{Locality} as part of our proof  of Schramm's locality conjecture for Bernoulli bond percolation \cite{MR2773031}, where it plays an important part in our ``uniformization'' of the methods of Contreras, Martineau, and Tassion \cite{contreras2021supercritical}.  A comparison of our results with  the previous literature is given at the end of this section. 

\medskip

Let us now state our result formally.
Let $G$ be a group with finite generating set $S$, so that $G\cong F_S / R$ for some normal subgroup $R$ of $F_S$. For each $n \geq 0$, let $R_n$ be the set of words of length at most $2^n$ in the free group $F_S$ that are equal to the identity in $G$, and let $\llangle R_n \rrangle$ be the normal subgroup of $F_S$ generated by $R_n$, so that the quotient map $F_S/\llangle R_n\rrangle \to G$ induces a covering map of the associated Cayley graphs that has injectivity radius at least $2^{n-1}-1$ (see \cref{lem:small_presentation_means_balls_look_similar}). We say that $(G,S)$ has a \textbf{new relation on scale n} if $\llangle R_{n+1} \rrangle \neq \llangle R_{n} \rrangle$. A finitely generated group $G$ is finitely presented if and only if it has a new relation on at most finitely many scales, so that the following theorem can indeed be thought of as stating that groups of polynomial growth are ``uniformly finitely presented".

\begin{thm}
\label{thm:main}
For each $K,k<\infty$ there exist constants $r_0=r_0(K)$ and $C=C(K,k)$ such that if $G$ is a group and $S$ is a finite generating set for $G$ with $|S|\leq k$ whose growth function $\operatorname{Gr}$ satisfies $\operatorname{Gr}(3r)\leq K \operatorname{Gr}(r)$ for some integer $r\geq r_0$ then
\[
\#\Bigl\{n\in \mathbb{N} : n\geq \log_2(r) \text{ and $(G,S)$ has a new relation on scale $n$} \Bigr\}\leq C.
\]
\end{thm}

\begin{remark}
Considering the abelian group $\prod_{i=1}^k(\Z/n_i \Z)$ with its standard generating set, where $n_1,\ldots,n_k$ are arbitrary, we see that it is not possible to control on \emph{which} scales we find a new relation; we only claim that the total \emph{number} of scales on which we find a new relation is bounded.
\end{remark}

We will prove the following theorem about the number of times we find an ``unexpected element'' during a breadth-first exploration of a (not necessarily normal) subgroup of a group of polynomial growth; we will see in \cref{sec:wrap_up} that this theorem easily implies \cref{thm:main}.

\begin{thm}[Breadth-first exploration of subgroups]
\label{thm:main_generalized}
For each $K$ and $k$ there exist constants $r_0=r_0(K)$ and $C=C(K,k)$ such that the following holds.
Let $G$ be a group with finite generating set $S$ satisfying $|S|\leq k$, let $H$ be a subgroup of $G$, and for each $n\geq 1$ let $H_n$ be the subgroup of $H$ generated by elements that have word length at most $2^n$ in $(G,S)$. If $r\geq r_0$ is such that $\Gr(3r)\leq K \Gr(r)$ then 
\[
\#\{n \geq \log_2 r :H_{n+1}\neq H_n\} \leq C.
\]
\end{thm}

\medskip

\begin{remark}
We believe that it should be possible to take the constants in \cref{thm:main,thm:main_generalized} to be independent of the size of the generating set. We do not pursue this here.
\end{remark}




\medskip

\noindent \textbf{Other previous results.} A more classical way to quantify the sense in which a presentation is finite is through \emph{Dehn functions}, \emph{filling length functions}, and so-called \emph{isoperimetric functions} (which do not refer to the same kind of isoperimetry mentioned in our above discussion of the structure vs.\ expansion dichotomy); see \cite{bridson2002geometry} for an overview and \cite{gersten2003isoperimetric} for results on nilpotent groups. As far as we can tell, however, the literature on these notions focusses on asymptotic properties of a fixed group and is not suitable for the kind of uniform-in-the-group results we wish to prove. Besides this, the notion of uniform finite presentation we consider is also rather different from these notions in terms of Dehn functions etc.

\medskip

There is a striking resemblance between our theorem and the following theorem of Tao \cite{tao2017inverse} (see also \cite[Appendix A]{tessera2017scaling}), which also relies on the structure theory of Breuillard, Green, and Tao.

\begin{thm}[{\!\!\cite{tao2017inverse}, Theorem 1.9}]
For each non-negative integer $d$ there exist constants $m_0$ and $C$ depending only on $d$ such that if $G$ is a group, $S$ is a finite, symmetric generating set for $G$ containing the identity and satisfying $|S^{nm}|\leq m^d |S^n|$ for some integers $n\geq 1$ and $m\geq m_0$ then there exists a continuous, piecewise-linear, non-decreasing function $f:[0,\infty)\to [0,\infty)$ with $f(0)=0$ that has at most $C$ pieces, each of which has slope equal to an integer bounded by $C$, such that
\[
\left|\log\frac{|S^{knm}|}{|S^{nm}|} - f(\log k)\right| \leq C 
\]
for every integer $k\geq 1$.
\end{thm}

Informally, this theorem states that, once we witness polynomial growth on a sufficiently large scale, the log-log plot of the growth function is well-approximated by a continuous, piecewise-linear function with bounded, integer valued slopes and a bounded number of ``kinks'' connecting the different pieces.

\medskip

Naively, one might hope that our bounded number of scales on which a new relation occurs are in correspondence with Tao's bounded number of scales on which the growth function has a ``kink'' in its log-log plot. Unfortunately this is not the case, at least when one allows generating sets of unbounded size: one can have a new relation without having a kink, and can have a kink without having a new relation. Indeed, as explained in \cite[Example 1.11]{tao2017inverse}, taking
\[
G = \begin{pmatrix}
1 & \Z & \Z \\ 0 & 1 & \Z \\ 0 & 0 & 1
\end{pmatrix} 
\qquad \text{ and } \qquad
S = \begin{pmatrix}
1 & [-N,N] & [-N^3,N^3] \\ 0 & 1 & [-N,N] \\ 0 & 0 & 1
\end{pmatrix}
\]
for a large integer $N$ yields
\[
\log \frac{|S^n|}{|S|} = \begin{cases} 3 \log n \pm O(1) & 1\leq n \leq N \\
4\log n - \log N \pm O(1) & n > N,
\end{cases} 
\]
so that this example's growth function has a kink at scale $\log_2 N$. On the other hand the pair $(G,S)$ does \emph{not} have a new relation at any at $k \geq 3$, and in particular does not have a new relation on the scale where it has a kink when $N$ is large.
Indeed, the relations of $(G,S)$ are generated by the
usual relations for the Heisenberg group
\[
\left[\biggl(\begin{smallmatrix}
1 & 1  & 0 \\ 0 & 1 & 0 \\ 0 & 0 & 1
\end{smallmatrix}\biggr), \biggl(\begin{smallmatrix}
1 & 0  & 0 \\ 0 & 1 & 1 \\ 0 & 0 & 1
\end{smallmatrix}\biggr)\right] = \biggl(\begin{smallmatrix}
1 & 0  & 1 \\ 0 & 1 & 0 \\ 0 & 0 & 1
\end{smallmatrix}\biggr),
\]
\[
\biggl(\begin{smallmatrix}
1 & 1  & 0 \\ 0 & 1 & 0 \\ 0 & 0 & 1
\end{smallmatrix}\biggr)
\biggl(\begin{smallmatrix}
1 & 0  & 1 \\ 0 & 1 & 0 \\ 0 & 0 & 1
\end{smallmatrix}\biggr) = \biggl(\begin{smallmatrix}
1 & 0  & 1 \\ 0 & 1 & 0 \\ 0 & 0 & 1
\end{smallmatrix}\biggr)
\biggl(\begin{smallmatrix}
1 & 1  & 0 \\ 0 & 1 & 0 \\ 0 & 0 & 1
\end{smallmatrix}\biggr),  
\qquad\text{and}\qquad
\biggl(\begin{smallmatrix}
1 & 0  & 0 \\ 0 & 1 & 1 \\ 0 & 0 & 1
\end{smallmatrix}\biggr) 
\biggl(\begin{smallmatrix}
1 & 0  & 1 \\ 0 & 1 & 0 \\ 0 & 0 & 1
\end{smallmatrix}\biggr) = \biggl(\begin{smallmatrix}
1 & 0  & 1 \\ 0 & 1 & 0 \\ 0 & 0 & 1
\end{smallmatrix}\biggr) \biggl(\begin{smallmatrix}
1 & 0  & 0 \\ 0 & 1 & 1 \\ 0 & 0 & 1
\end{smallmatrix}\biggr)
\]
together with the following three sets of relations relating the extra generators in $S$ to the standard generators:
\[
\left \{\biggl(\begin{smallmatrix}
1 & a \pm  1 & c \\ 0 & 1 & b \\ 0 & 0 & 1
\end{smallmatrix}\biggr) = \biggl(\begin{smallmatrix}
1 & a  & c \\ 0 & 1 & b \\ 0 & 0 & 1
\end{smallmatrix}\biggr) \biggl(\begin{smallmatrix}
1 & \pm 1  & 0 \\ 0 & 1 & 0 \\ 0 & 0 & 1
\end{smallmatrix}\biggr) : a,b \in [-N,N], c \in [-N^3,N^3] \right\},
\]
\[
\left \{\biggl(\begin{smallmatrix}
1 & a  & c \\ 0 & 1 & b \pm 1 \\ 0 & 0 & 1
\end{smallmatrix}\biggr) = \biggl(\begin{smallmatrix}
1 & 0  & 0 \\ 0 & 1 & \pm 1 \\ 0 & 0 & 1
\end{smallmatrix}\biggr)\biggl(\begin{smallmatrix}
1 & a  & c \\ 0 & 1 & b \\ 0 & 0 & 1
\end{smallmatrix}\biggr)  : a,b \in [-N,N], c \in [-N^3,N^3] \right\},
\]
\[
\left \{\biggl(\begin{smallmatrix}
1 & a  & c \pm 1 \\ 0 & 1 & b \\ 0 & 0 & 1
\end{smallmatrix}\biggr) = \biggl(\begin{smallmatrix}
1 & a  & c \\ 0 & 1 & b \\ 0 & 0 & 1
\end{smallmatrix}\biggr) \biggl(\begin{smallmatrix}
1 & 0  & \pm 1 \\ 0 & 1 & 0 \\ 0 & 0 & 1
\end{smallmatrix}\biggr) : a,b \in [-N,N], c \in [-N^3,N^3] \right\}.
\]
These relations all have word length at most five in $(G,S)$, so that the example has the desired properties.
 Conversely, taking the direct product $G \times \Z/N$ with generating set $S \times \{-1,0,1\}$ yields an example where there is a new relation at scale $\log_2 N$ but where the growth function does not have any kinks.

\medskip

\noindent \textbf{About the proof.} It is natural to describe the proof of \cref{thm:main,thm:main_generalized} ``backwards'', as a sequence of reductions, although we have written it ``forwards'' as a sequence of extensions and generalizations. In this backwards description, the ``first'' step (which is the last part of the paper) is to reduce from groups of polynomial growth to  nilpotent groups of bounded step using the Breuillard-Green-Tao theorem. Next, this statement about nilpotent groups is in turn reduced to an analogous statement about breadth-first exploration of discrete subgroups in a \emph{Carnot group}, a simply connected nilpotent Lie group carrying the additional structure of a \emph{stratification} and  \emph{homogeneous left-invariant metric}. Nilpotent groups are related to Carnot groups for example by \emph{Pansu's theorem} \cite{pansu1983croissance,breuillard2013rate}, which states roughly that the large-scale geometry of a finitely generated nilpotent group is well-modelled by an appropriate Carnot group equipped with a left-invariant homogeneous metric. Finally, this statement about Carnot groups is reduced to a statement about \emph{vector spaces} using the close connection between the discrete subgroups of a simply connected nilpotent Lie group and the additive bracket-closed subgroups of its associated Lie algebra. This step of the reduction is the most involved part of the paper, with the connection between additive and multiplicative lattices being developed at length in \cref{sec:nilpotent_lattices}. This ends the chain of reductions, and leaves us with a problem we must actually solve directly: Bounding the number of times we find an ``unexpected element'' of a discrete subgroup of $\R^d$ as we explore the subgroup with an increasing family of convex, symmetric sets. This is done in \cref{sec:Abelian} as an application of Minkowski's second theorem, a classical result in the geometry of numbers. 

\medskip

Let us stress again that we have described the argument here in the opposite order to the way we carry it out, so that the result about subgroups of $\R^d$ is the first thing we prove.

\medskip

\noindent \emph{Disclaimer:} Neither author is an expert in approximate groups, Lie theory, or the geometry of numbers. As such, it is likely that we have included a larger amount of detail in the proofs than would be considered necessary by experts, or have re-derived known results from scratch. While we have attempted to provide appropriate attribution to the intermediate results of the paper as much as possible, we would be happy to receive comments and corrections from experts.

\begin{remark}
  Since the present paper first appeared, two relevant papers by Tessera and Tointon have since appeared. The first paper \cite{tessera2024ballsgroupsvolumestructure} establishes a stronger version of \cref{thm:main} as a corollary, with optimal bounds on the number of scales at which new relations appear that do not depend on the degree. The proof is very different than that given here. The second paper \cite{tessera2023smalldoublingimpliessmall} allows us to replace the ``small tripling'' hypothesis $\Gr(3r)\leq K \Gr(r)$ in \cref{thm:main,thm:main_generalized} with the more natural hypothesis of ``small doubling'' $\Gr(2r)\leq K \Gr(r)$ because, as the title of the paper says: \emph{Small doubling implies small tripling on large scales}.
\end{remark}

\section{Background on nilpotent groups and Lie groups}

In this section we review the relevant background material and establish some notational conventions. We have included a rather thorough account of the basic theory with the hope that our paper can be easily understood by probabilists. 

\medskip

Given a group $G$, the \textbf{commutator} of two elements $x,y\in G$ is defined by $[x,y]=xyx^{-1}y^{-1}$. The \textbf{lower central series} of $G$ is defined recursively by $G_1=G$ and $G_{i+1}=[G_i,G]$ for each $i\geq 1$, where we write $[A,B]:=\{[a,b]:a\in A, b\in B\}$ for subsets $A$ and $B$ of $G$. The group $G$ is said to be \textbf{nilpotent} if $G_{i+1}=\{\mathrm{id}\}$ for all sufficiently large $i$, with the minimal such $i$ denoted by $s$ and known as the \textbf{step} of $G$. A group is said to be \textbf{virtually nilpotent} if it has a nilpotent subgroup of finite index. Given $s\geq 1$ and a set $S$, the \textbf{free step $s$ nilpotent group} $N_{s,S}$ is defined to be the quotient of the free group $F_S$ by the step-$s$ nilpotency relations, which state that all iterated commutators of length at least $s+1$ are equal to the identity. 
 The free step $s$ nilpotent group $N_{s,S}$ can also be defined up to unique $S$-preserving isomorphism by the universal property that it is nilpotent of step at most $s$, contains $S$, and every function from $S$ to a nilpotent group of step at most $s$ can be uniquely extended to a homomorphism from $N_{s,S}$ to that group.

\subsection{(Nilpotent) Lie groups and the Baker-Campbell-Hausdorff formula}

Recall that a (real) \textbf{Lie group} is a group that is also a finite-dimensional real smooth manifold, in such a way that the group operations of multiplication and inversion are smooth maps $G\times G \to G$ and $G\to G$. By Gleason, Montgomery, and Zippin's solution to Hilbert's fifth problem \cite{MR49204,MR49203}, one can equivalently define a Lie group as a group that is also a finite-dimensional \emph{topological} manifold with continuous multiplication and inversion operations; such a group carries a unique smooth structure compatible with its algebraic structure. More generally, Yamabe \cite{MR36766} proved that every locally compact, connected topological group is a projective limit of Lie groups (possibly of divergent dimension). These facts underlie the ubiquity of Lie groups in the scaling limit theory of discrete groups, and in particular are used directly in Gromov's original proof of his polynomial growth theorem \cite{gromov81poly}. Further background on these topics can be found in \cite{MR3237440}. For our purposes, Lie groups become relevant primarily via a theorem of Pansu, which allows us to approximate the balls in the Cayley graph of a nilpotent group in terms of the balls in a certain \emph{left-invariant homogeneous metric on a Carnot group}; this is explained in \cref{subsec:Carnot_background}.

\medskip

\textbf{Lie algebras}. A \textbf{Lie algebra} $\mathfrak{g}$ is a vector space equipped with a binary operation, the Lie bracket $[\cdot, \cdot]: \mathfrak{g} \times \mathfrak{g} \rightarrow \mathfrak{g}$, that is bilinear, antisymmetric ($[X, Y] = -[Y, X]$ for all $X, Y \in \mathfrak{g}$), and satisfies the Jacobi identity ($[X, [Y, Z]] + [Y, [Z, X]] + [Z, [X, Y]] = 0$ for all $X, Y, Z \in \mathfrak{g}$). A subset $A$ of a Lie algebra is said to be \textbf{bracket-closed} if $[X,Y]\in A$ for every $X,Y\in A$.
For ease of reading, we will loosely follow the convention that points in a Lie algebra are denoted using upper-case letters, while points in a Lie group are denoted using lower-case letters. We will also assume without further comment that all Lie algebras are finite-dimensional. 

\medskip

To each Lie group $G$, we can associate a Lie algebra $\mathfrak{g}$ arising from the tangent space at the identity; the details of this construction are not important to us and can be found in any textbook on the subject. In the concrete case that $G$ is a Lie subgroup of a general linear group $\mathrm{GL}_n$ for some $n\geq 1$, the affine space
  $I+\mathfrak{g}$ is precisely the tangent space at the identity to $G$ in the space of all $n\times n$ matrices, so that $\mathfrak{g}$ is a Lie subalgebra (i.e.\ a bracket-closed linear subspace) of the Lie algebra $\mathfrak{gl}_n$ of all $n\times n$ matrices with Lie bracket defined by the commutator $[X,Y]=XY-YX$. (In particular, $\mathfrak{gl}_n$ is the Lie algebra associated to the Lie group $\mathrm{GL}_n$.)
%
%
In fact this case is not particularly special:
Ado's theorem states that every Lie algebra is isomorphic to a Lie subalgebra  of $\mathfrak{gl}_n$ for some $n\geq 1$ \cite[Chapter 2.3]{MR3237440}. (Ado's theorem does \emph{not} imply that every Lie \emph{group} is isomorphic to a Lie subgroup of a general linear group, although it does imply a ``local'' version of the same claim.)

\medskip

The fundamental theorems of Lie (see e.g.\ \cite[Chapter 2.5.1]{MR3237440}) state in particular that there is a one-to-one correspondence between (isomorphism classes of) Lie algebras and \emph{simply connected} Lie groups. On the other hand, Lie groups that are connected but not simply connected have universal covers which are simply connected Lie groups with the same Lie algebra. 

\medskip

The \textbf{lower central series} of the Lie algebra $\mathfrak{g}$ is defined recursively by $\mathfrak{g}_1=\mathfrak{g}$ and $\mathfrak{g}_{i+1}=[\mathfrak{g}_i,\mathfrak{g}]$ for each $i\geq 1$, where if $A$ and $B$ are two subsets of $\mathfrak{g}$ then we write $[A,B]=\{[a,b]:a\in A,b\in B\}$. It is a simple consequence of the Jacobi identity that $[\mathfrak{g}_i,\mathfrak{g}_j] \subseteq \mathfrak{g}_{i+j}$ for every $i,j \geq 1$ \cite[Lemma 1.4.3]{Wang2023}. The Lie algebra $\mathfrak{g}$ is said to be \textbf{nilpotent} if $\mathfrak{g}_{i+1}=\{0\}$ for all sufficiently large $i$; the minimal such $i$ is called the \textbf{step} of $\mathfrak{g}$ and is usually denoted $s$. Nilpotence of a Lie \emph{group} is defined as for any other group (meaning that the lower central series terminates at the identity subgroup); a connected Lie group is nilpotent if and only if its corresponding Lie algebra is nilpotent.

\medskip

\textbf{Lie polynomials.} 
A \textbf{Lie monomial} of degree $d$ in the terms $X_1, X_2, \ldots, X_n \in \mathfrak{g}$ is an expression obtained by taking iterated Lie brackets of these terms in some way, so that the sum over $i$ of the total number of times $X_i$ appears is $d$. In other words, a Lie monomial of degree $1$ in the terms $X_1,\ldots,X_n$ is an expression of the form $L(X_1,\ldots,X_n)=X_i$ for some $1\leq i \leq n$, while each Lie monomial of degree $d$ can be written $L(X_1,\ldots,X_n)=[L_1(X_1,\ldots,X_n),L_2(X_1,\ldots,X_n)]$ for some Lie monomials $L_1$ and $L_2$ whose degrees sum to $d$. 
A \textbf{Lie polynomial} $P(X_1, X_2, \ldots, X_n)$ in elements $X_1, X_2, \ldots, X_n \in \mathfrak{g}$ is a linear combination of Lie monomials; it is said to be \textbf{homogeneous of degree $d$} if every Lie monomial in the linear combination has degree $d$. Thus, homogeneous Lie polynomials of degree $d$ obey the scaling transformation $P(\lambda X_1,\ldots, \lambda X_n)=\lambda^d P(X_1,\ldots, X_n)$. 

\medskip

\textbf{The exponential map.} Given a Lie group $G$ and associated Lie algebra $\mathfrak{g}$, there is a canonically defined \textbf{exponential map} $\exp:\mathfrak{g}\to G$, which for Lie subgroups of $\mathrm{GL}_n$ coincides with ordinary matrix exponentiation. (We will omit the general definition of the exponential map; everything we need to know about it will be captured by the Baker-Campbell-Hausdorff formula.) The exponential map is smooth, and is a diffeomorphism in a neighbourhood of the identity, but might not be injective or surjective.
The \textbf{Baker-Campbell-Hausdorff (BCH) formula} \cite[Chapter 1.2.5]{MR3237440} states that if $G$ is a Lie group with Lie algebra $\mathfrak{g}$ then there exists an open neighbourhood $U$ of the origin such that
\begin{align*}
\log \left[\exp X \exp Y\right] = X + Y + \frac{1}{2}[X,Y]+\frac{1}{12}[X,[X,Y]]-\frac{1}{12}[Y,[X,Y]]+\cdots
= \sum_{i=1}^\infty L_i(X,Y)
\end{align*}
for all $X,Y \in U$, where each $L_i$ is a homogeneous Lie polynomial of degree $i$ with rational coefficients and where we write $\log$ for the inverse of the exponential map on $\exp(U)$. (The Lie polynomials $L_i$ appearing in the BCH formula are universal and do not depend on the choice of Lie group $G$.) When $G$ is a simply connected nilpotent Lie group, all terms with $i>s$ are identically zero, the exponential function $\exp:\mathfrak{g}\to G$ is defined globally, and the BCH formula holds for \emph{all} $X,Y \in\mathfrak{g}$.
This lets us define the \textbf{BCH product} on the Lie algebra $\mathfrak{g}$ as
\[X\diamond Y := \log \left[e^Xe^Y\right]  = X + Y + \frac{1}{2}[X,Y]+\frac{1}{12}[X,[X,Y]]-\frac{1}{12}[Y,[X,Y]]+\cdots,
\]
so that $(\mathfrak{g},\diamond)$ is isomorphic to $G$ as a Lie group. Note that $(aX) \diamond (bX)=(a+b)X$ for every $X \in \mathfrak{g}$ and $a,b\in \R$, so that $0$ is both the additive and BCH identity and that $(-X)$ is both the additive and BCH inverse of $X$.

\begin{example}
\label{ex:heisenberg_def}
The Heisenberg group $H$ and its Lie algebra $\mathfrak{h}$ are given by
\[
H=\left(\begin{matrix} 1 & \R & \R \\ 0 & 1 & \R \\ 0&0&1\end{matrix}\right) \qquad \text{ and } \qquad
\mathfrak{h}=\left(\begin{matrix} 0 & \R & \R \\ 0 & 0 & \R \\ 0&0&0\end{matrix}\right),
\]
with exponential map and logarithm
\[
\exp \left(\begin{matrix} 0 & a & c \\ 0 & 0 & b \\ 0&0&0\end{matrix}\right) = \left(\begin{matrix} 1 & a & c+\frac{ab}{2} \\ 0 & 1 & b \\ 0&0&1\end{matrix}\right) \qquad \text{ and } \qquad \log \left(\begin{matrix} 1 & a & c \\ 0 & 1 & b \\ 0&0&1\end{matrix}\right) = \left(\begin{matrix} 0 & a & c-\frac{ab}{2} \\ 0 & 0 & b \\ 0&0&0\end{matrix}\right).
\]
Note that the exponential map here is just the usual matrix exponential, which for $X\in \mathfrak{h}$ satisfies $e^X=I+X+\frac{1}{2}X^2$. The Lie bracket on $\mathfrak{h}$ is given by
\[
\left[\left(\begin{matrix} 0 & a & c \\ 0 & 0 & b \\ 0&0&0\end{matrix}\right),\left(\begin{matrix} 0 & x & z \\ 0 & 0 & y \\ 0&0&0\end{matrix}\right)\right] = \left(\begin{matrix} 0 & 0 & ay-xb \\ 0 & 0 & 0 \\ 0&0&0\end{matrix}\right).
\]
Since $H$ is step-$2$ nilpotent, the BCH multiplication on $\mathfrak{h}$ is given by
\begin{align*}
\left(\begin{matrix} 0 & a & c \\ 0 & 0 & b \\ 0&0&0\end{matrix}\right)\diamond \left(\begin{matrix} 0 & x & z \\ 0 & 0 & y \\ 0&0&0\end{matrix}\right) 
&=\left(\begin{matrix} 0 & a & c \\ 0 & 0 & b \\ 0&0&0\end{matrix}\right) + \left(\begin{matrix} 0 & x & z \\ 0 & 0 & y \\ 0&0&0\end{matrix}\right) +  \frac{1}{2}\left[\left(\begin{matrix} 0 & a & c \\ 0 & 0 & b \\ 0&0&0\end{matrix}\right),\left(\begin{matrix} 0 & x & z \\ 0 & 0 & y \\ 0&0&0\end{matrix}\right)\right] \\
&=\left(\begin{matrix} 0 & a+x & c+z+\frac{ay-bx}{2} \\ 0 & 0 & b+y \\ 0&0&0\end{matrix}\right).
\end{align*}
We will see in \Cref{ex:heisenberg_subgroups} that the non-integer rational coefficient $1/2$ appearing in this expression leads to complications when comparing lattices in $H$ and $\mathfrak{h}$.
\end{example}

\begin{remark}
It follows from the BCH formula that
\[
\log [e^X e^Y e^{-X}e^{-Y}] 
= [X,Y]+\frac{1}{2}[X,[X,Y]]+\frac{1}{2}[Y,[X,Y]]+\cdots
\]
for $X$ and $Y$ in a neighbourhood of the origin in $\mathfrak{g}$ (all of $\mathfrak{g}$ when $G$ is nilpotent and simply connected). As such, the BCH commutator and the Lie bracket agree to first order as $X,Y\to 0$, but are not exactly equal unless $G$ is nilpotent of step at most $2$.
\end{remark}

We will also make use of the \textbf{Zassenhaus formula} \cite{casas2012efficient}, a dual form of the BCH formula which states in particular that if $G$ is $s$-step nilpotent then
\begin{align}
e^{X+Y}&= e^{X} e^{Y} e^{-\frac{1}{2} [X,Y]} 
e^{\frac{1}{6}(2[Y,[X,Y]]+ [X,[X,Y]] )} 
e^{\frac{-1}{24}([[[X,Y],X],X] + 3[[[X,Y],X],Y] + 3[[[X,Y],Y],Y]) } \cdots\nonumber\\
&= e^Xe^Y e^{\tilde L_2(X,Y)}e^{\tilde L_3(X,Y)} \cdots e^{\tilde L_s(X,Y)}
\end{align}
for every $X,Y\in \mathfrak{g}$, where $\tilde L_i$ is a homogeneous Lie polynomial of degree $i$ with rational coefficients for each $i\geq 2$.

\medskip

 \textbf{Subgroups and lattices.}
Let $G$ be a simply connected nilpotent Lie group with Lie algebra~$\mathfrak{g}$. For each set $A\subseteq \mathfrak{g}$, we define $\mathscr{L}(A)$ to be the smallest Lie subalgebra of $\mathfrak{g}$ containing $A$. The exponential map identifies closed connected subgroups of $G$ with Lie subalgebras of $\mathfrak{g}$, so that every closed connected subgroup of $G$ is itself a simply connected nilpotent Lie group. (For general Lie groups, the image under the exponential map of a Lie subalgebra might not be closed, but for simply connected nilpotent groups it is always closed since the exponential map is a diffeomorphism.) As such, for each subset $A$ of $G$, the intersection of all closed connected subgroups of $G$ containing $A$ is a closed connected subgroup of $G$ that is equal to $\exp(\mathscr{L}(\log A))$. We write $\mathscr{C}(A)$ for this minimal closed connected subgroup of $G$ containing $A$.

\begin{thm}[Mal'cev]
If $G$ is a simply connected nilpotent Lie group and $H$ is a closed subgroup of $G$ then $G/H$ is compact if and only if $H$ is not contained in any proper closed connected subgroup of $G$. In particular, $\mathscr{C}(H)/H$ is compact.
\end{thm}

(The quotients $G/H$ and $\mathscr{C}(H)/H$  appearing here are topological spaces, and do not carry group structures in general.)
We call a subgroup $\Gamma$ of a simply connected nilpotent Lie group $G$ a \textbf{lattice} in $G$ if it is discrete with compact quotient $G/\Gamma$. (Note that discrete subgroups of Hausdorff topological groups are automatically closed.)

\begin{remark}
For a general Lie group, a lattice is defined to be a discrete subgroup for which the quotient admits a finite left-invariant measure (a.k.a.\ Haar measure); for simply connected nilpotent Lie groups this is equivalent to $G/\Gamma$ being compact by Mal'cev's theorem.
This theorem also gives several further characterisations of a discrete subgroup being a lattice that we omit since we do not use them.
\end{remark}

A further theorem of Mal'cev \cite[Theorem 2.12]{raghunathan1972discrete} states that a simply connected nilpotent Lie group $G$ admits a lattice if and only if its Lie algebra $\mathfrak{g}$ admits a basis $e_1,\ldots,e_d$ for which the \textbf{structure constants} $(T_{i,j}^k)_{i,j,k}$, defined by $[e_i,e_j]=\sum_k T_{i,j}^k e_k$,  are rational. It is a consequence of this theorem \cite[Proposition 2.3.7]{Wang2023} that if $\Gamma$ is a lattice in a simply connected nilpotent Lie group then there exist (additive) lattices $\Lambda^-$ and $\Lambda^+$ in $\mathfrak{g}$ such that $\Lambda^- \subseteq \log \Gamma \subseteq \Lambda^+$. We will prove significantly stronger versions of this fact in \cref{sec:nilpotent_lattices}.

\subsection{Carnot groups and Pansu's theorem}
\label{subsec:Carnot_background}

A \textbf{Carnot group} is a simply connected nilpotent Lie group $G$ of some step $s$ whose Lie algebra $\mathfrak{g}$ is equipped with a decomposition
\[
\mathfrak{g}=V_1 \oplus V_2 \cdots \oplus V_s
\]
for some non-trivial linear subspaces $V_1,\ldots,V_s$ such that $[V_1,V_i]=V_{i+1}$ for every $1\leq i <s$ and $[V_1,V_s]=\{0\}$. 
It can be shown that
for such a decomposition we have moreover that $[V_i,V_j]\subseteq V_{i+j}$ for every $i,j \geq 1$, where $V_{i+j}:=\{0\}$ for $i+j>s$. A decomposition $\mathfrak{g}=V_1\oplus \cdots \oplus V_s$ satisfying these conditions is known as a \textbf{stratification} of $\mathfrak{g}$; the subspace $V_1$, which generates $\mathfrak{g}$ as a Lie algebra, is known as the \textbf{horizontal subspace}. Not every nilpotent lie algebra admits a stratification. While a Carnot group consists of both a nilpotent Lie group and a choice of stratification of its Lie algebra, we will nevertheless write e.g.\ ``let $G$ be a Carnot group'' when this does not cause confusion.

\medskip

 Given a Carnot group $G$ and a real number $\lambda>0$ the \textbf{dilation maps} $\delta_\lambda:\mathfrak{g}\to\mathfrak{g}$ and $D_\lambda:G\to G$ are defined by
\[
\delta_\lambda(x_1+x_2+\cdots+x_s)= \lambda x_1 + \lambda^2 x_2 + \cdots \lambda^s x_s \qquad \text{ and } \qquad D_\lambda(x) = \exp(\delta_\lambda(\log(x))),
\] 
where we write $x=x_1+x_2+\cdots + x_s$ for the decomposition of $x\in \mathfrak{g}$ associated to the stratification $\mathfrak{g}=V_1\oplus V_2 \oplus \cdots \oplus V_s$. These dilation maps satisfy the semigroup properties $\delta_{\lambda\mu}=\delta_\lambda \delta_\mu=\delta_\mu\delta_\lambda$ and $D_{\lambda \mu}=D_\lambda D_\mu=D_\mu D_{\lambda}$ for every $\lambda,\mu>0$.
It is a consequence of the Baker-Campbell-Hausdorff formula and the fact that $[V_i,V_j] \subseteq V_{i+j}$ that $D_\lambda$ is a Lie group automorphism of $G$ for every $\lambda>0$. A metric $d:G\times G\to [0,\infty)$ on a Carnot group is said to be \textbf{left-invariant and homogeneous} if
\[
d(zx,zy) = d(x,y) \qquad \text{and} \qquad d(D_\lambda x,D_\lambda y) = \lambda d(x,y) 
\]
for every $x,y,z\in G$ and $\lambda>0$. Note that the abelian group $\R^d$ is a Carnot group with $V_1=\R^d$, and the left-invariant homogeneous metrics on $\R^d$ seen as a Carnot group are equivalent to norms on $\R^d$. In general, left-invariant homogeneous metrics are closely analogous to norms, and also have the property that they are determined by the unit ball $B$ around the origin: 
\[
d(x,y)= \inf\{\lambda: x^{-1}y \in D_\lambda(B)\}.
\]
In particular, the ball $\{x:d(0,x)\leq \lambda\}$ is equal to $D_\lambda(B)$. 

\medskip

We now introduce the free step-$s$ nilpotent Lie algebra and the free step-$s$ Carnot group, referring the reader to \cite[Chapter 14.1]{bonfiglioli2007stratified} for proofs the objects we discuss here are well-defined.
Let $S$ be a finite set. The \textbf{free step-$s$ nilpotent Lie algebra} $\mathfrak{f}_{s,S}$ is defined to be the unique-up-to-$S$-preserving-isomorphism nilpotent Lie algebra of step $s$ that is generated by $S$ and is such that if $\mathfrak{g}$ is any nilpotent Lie algebra of step at most $s$ and $\phi:S\to \mathfrak{g}$ is any function, then there exists a unique Lie algebra homomorphism $\mathfrak{f}_{s,S}\to \mathfrak{g}$ extending $\phi$. The 
 Lie algebra $\mathfrak{f}_{s,S}$ may be
   equipped with a canonical stratification defined in terms of \emph{Hall bases}, making its associated BCH-multiplication Lie group $G_{s,S}=(\mathfrak{f}_{s,S},\diamond)$ into a Carnot group known as the \textbf{free step-$s$ Carnot group} or \textbf{free step-$s$ nilpotent Lie group} over $S$; see \cite[Chapters 14.1 and 14.2]{bonfiglioli2007stratified} for details. (It is convenient to define the free step-$s$ nilpotent Lie group over $S$ via BCH multiplication so that it contains the set $S$.) Moreover, this stratification $\mathfrak{f}_{s,S}=V_1 \oplus \cdots \oplus V_s$ has the property that $V_1$ is equal to the linear span of $S$. Finally, if we define $\Gamma_{s,S}$ to be the subgroup of $G_{s,S}$ generated by $S$, then $\Gamma_{s,S}$ is a lattice in $G_{s,S}$ that is isomorphic (via an $S$-preserving isomorphism) to the discrete free step-$s$ nilpotent group $N_{s,S}$. Indeed, the fact that $\Gamma_{s,S}$ is discrete in $G_{s,S}$ can be proven using Mal'cev's theorem on rational structure constants, since the structure constants in the Hall basis are all equal to $1$, while $\Gamma$ is a lattice since $S$ generates $G_{s,S}$ as a Lie group. (The fact that $\Gamma_{s,S}$ is discrete can also be proven using the techniques of \cref{sec:nilpotent_lattices}.) Finally, the fact that $\Gamma_{s,S}$ is isomorphic to $N_{s,S}$ can be deduced straightforwardly from the relevant universal properties since every torsion-free finitely generated nilpotent group can be embedded as a lattice in a nilpotent Lie group, sometimes known as the \textbf{Mal'cev completion} of the group \cite[Theorem 2.18]{raghunathan1972discrete}.

\medskip

In his thesis \cite{pansu1983croissance}, Pansu proved that Cayley graphs of finitely generated nilpotent groups converge under rescaling to Carnot groups equipped with certain left-invariant homogeneous metrics known as \emph{sub-Finsler metrics}. (We will not need the definition of sub-Finsler metrics in this paper.) Note that the Carnot group arising in this limit might \emph{not} be isomorphic to the Mal'cev completion of the relevant nilpotent group, and indeed the Mal'cev completion might not admit a stratification. However, if $N$ is a torsion-free nilpotent group with finite generating set $S$ and the lie algebra $\mathfrak{g}$ of the Mal'cev completion $G$ happens to be simply connected and admit a stratification $\mathfrak{g}=V_1\oplus \cdots \oplus V_s$ with $\log S \subseteq V_1$, then there exists a left-invariant homogeneous metric $d_G$ on $G$ such that
\begin{equation}
\label{eq:Pansu}
d_S(x,y)=(1\pm o(1)) d_G(x,y) \qquad \text{ as $d_S(x,y)\to \infty$},
\end{equation}
where $d_S$ denotes the word metric on $N$. It follows in particular that $(N,\frac{1}{n}d_S)$ converges to $(G,d_G)$ in the Gromov-Hausdorff sense as $n\to\infty$. See \cite{breuillard2013rate,tashiro2022speed} for details and quantitative refinements of this theorem.

\section{Exploring abelian lattices with convex sets}
\label{sec:Abelian}

In this section we prove the following theorem, which will eventually be used to prove our main theorems by reduction to the abelian case.

\begin{thm}
\label{thm:Abelian_convex}
Let $d\geq 1$, let $\Lambda$ be a discrete subgroup of $\R^d$, and let $K_1 \subseteq K_2 \subseteq \cdots$ be an increasing sequence of non-empty, symmetric, convex sets in $\R^d$. For each $n\geq 1$ let $\Lambda_n = \spanZ(\Lambda \cap K_n)$ so that $\Lambda_1 \subseteq \Lambda_2 \subseteq \cdots$ is an increasing sequence of subgroups of $\Lambda$. Then 
\[
\#\{n :\Lambda_{n+1}\neq \Lambda_n\} \leq d +1 + \sum_{\ell=1}^d \lfloor \log_2 \ell! \rfloor.
\]
\end{thm}

\begin{remark}
By Stirling's formula, the upper bound appearing here is asymptotic to $\frac{1}{2}d^2 \log_2 d$ as $d\to \infty$. We have not investigated the optimality of this bound.
\end{remark}

We will deduce this theorem as a consequence of Minkowski's second theorem \cite[p. 203]{MR1434478}, which states that if $\Lambda$ is a lattice in $\R^d$ with $d\geq 1$ and $K$ is a non-empty, symmetric convex subset of $\R^d$ then
\begin{equation}
\label{eq:Minkowski}
 1 \leq \frac{\operatorname{vol}(\R^d/\Lambda)}{2^d \operatorname{vol}(K) \prod_{i=1}^d \lambda_i(\Lambda,K)} \leq d!,
\end{equation}
where 
\[
\lambda_i(\Lambda,K)=\inf\bigl\{\lambda>0: \lambda K \cap \Lambda \text{ contains at least $i$ linearly independent vectors}\bigr\}
\]
for each $1\leq i \leq d$. 
For our purposes, the most important feature of \eqref{eq:Minkowski} is that the expression $2^d \operatorname{vol}(K) \prod_{i=1}^d \lambda_i(\Lambda,K)$ is determined by $K$ and $\Lambda \cap K$ whenever $\Lambda \cap K$ has real span equal to $\R^d$. Indeed, if $\Lambda$ is a lattice in $\R^d$ then every fundamental domain for $\Lambda$ has volume $\operatorname{vol}(\R^d/\Lambda)$ and if $\Lambda_1 \subseteq \Lambda_2$ are two lattices in $\R^d$ then $\Lambda_1$ is a finite-index subgroup of $\Lambda_2$ with index
\[
[\Lambda_2:\Lambda_1] = \frac{\operatorname{vol}(\R^d/\Lambda_1)}{\operatorname{vol}(\R^d/\Lambda_2)}.
\] Thus, if $\Lambda_1 \subseteq \Lambda_2$ are two lattices and $K$ is a symmetric convex set such that $\Lambda_1 \cap K$ and $\Lambda_2 \cap K$ are equal and both have real linear span equal to $\R^d$ then Minkowski's second theorem implies that~$[\Lambda_2:\Lambda_1] \leq d!$.

\begin{proof}[Proof of \cref{thm:Abelian_convex}]
It suffices to prove that if $K_1\subseteq K_2\subseteq \cdots \subseteq K_n$ is an increasing sequence of non-empty, symmetric convex subsets of $\R^d$ and $\Lambda$ is a discrete subgroup of $\R^d$ such that the subgroups $\Lambda_i:=\spanZ( \Lambda \cap K_i)$ satisfy $\Lambda_{i+1}\neq \Lambda_i$ for every $1\leq i < n$ then $n \leq d+1+\sum_{\ell=0}^d \lfloor \log_2 \ell! \rfloor$. 
We may also assume without loss of generality that $\Lambda=\Lambda_n$ is a lattice in $\R^d$, replacing $\Lambda$ with $\Lambda_n$ and $\R^d$ with the subspace spanned by $\Lambda_n$ otherwise.
Fix such a pair $\Lambda$ and $(K_1,\ldots,K_n)$ and for each $0\leq \ell \leq d$ let $i_\ell$ be minimal such that the real span of $\Lambda_{i_\ell}$ has dimension at least $\ell$, setting $i_{d+1}=n+1$ for notational convenience. 
Since $\Lambda_2\neq \Lambda_1$, we must have that $i_0=1$ and $i_1 \in \{1,2\}$. For each $1\leq \ell \leq d$ define $V_\ell$ to be the real span of $\Lambda_{i_\ell}$, so that $\Lambda_j$ is a lattice in $V_\ell$ for every $1\leq \ell \leq d$ and $i_\ell \leq j <i_{\ell+1}$. (Note that $V_\ell$ might have dimension strictly larger than $\ell$, in which case $i_{\ell+1}=i_\ell$.) Suppose that $0\leq \ell \leq d$ is such that $i_{\ell+1} \geq i_\ell + 2$. In this case, the subgroups $\Lambda_{i_\ell}$ and $\Lambda_{i_{\ell+1}-1}$ are both lattices in $V_\ell$ with $\Lambda_{i_\ell} \cap K_{i_\ell}=\Lambda_{i_{\ell+1}-1} \cap K_{i_\ell}$ and with $\Lambda_{i_\ell}\cap K_{i_\ell}$ having real span equal to $V_\ell$. As such, it follows by Minkowski's second theorem that
\[
[ \Lambda_{i_{\ell+1}-1}: \Lambda_{i_\ell}] = \frac{\operatorname{vol}(V_\ell/\Lambda_{i_\ell})}{\operatorname{vol}(V_\ell/\Lambda_{i_{\ell+1}-1})} \leq \ell!
\]
for each $1\leq \ell \leq d$ such that $i_{\ell+1} \geq i_\ell+2$.
Now, using that
\[
[ \Lambda_{i_{\ell+1}-1}: \Lambda_{i_\ell}] = \prod_{j=i_\ell}^{i_{\ell+1}-2} [ \Lambda_{j+1}: \Lambda_{j}] \geq 2^{i_{\ell+1}-1-i_\ell}
\]
it follows that $i_{\ell+1}-i_\ell \leq 1+\lfloor\log_2 \ell! \rfloor$
for every $1\leq \ell \leq d$ such that $i_{\ell+1} \geq i_\ell+2$. Since the same inequality also holds trivially when $i_{\ell+1} < i_\ell+2$, it follows that
\[
n = i_{d+1}-i_0 = \sum_{\ell=0}^d (i_{\ell+1}-i_\ell) \leq d+1 + \sum_{\ell=1}^d \lfloor \log_2 \ell!\rfloor
\]
as claimed.
\end{proof}

\section{Additive and multiplicative subgroups of nilpotent Lie algebras}
\label{sec:nilpotent_lattices}

Our goal in this section is to clarify the relationship between lattices in a simply connected nilpotent Lie group and its associated Lie algebra.
As mentioned above, closed, connected subgroups of a simply connected nilpotent Lie group $G$ are in bijection with Lie subalgebras of the Lie algebra $\mathfrak{g}$, which are precisely the closed, connected, bracket-closed additive subgroups of $\mathfrak{g}$. Without the assumption of connectivity, the exponential map need not interact this nicely with the subgroup structure of $G$: It is possible to have subgroups $H$ of $G$ for which $\log H$ is not an additive subgroup of $\mathfrak{g}$ and to have bracket-closed additive subgroups of $\mathfrak{g}$ whose image under the exponential is not a subgroup of $G$. 

\begin{example}
\label{ex:heisenberg_subgroups}
Let the Heisenberg group $H$ and its Lie algebra $\mathfrak{h}$ be as in \Cref{ex:heisenberg_def}. The set of all elements of $H$ whose matrix entries are integers is a lattice in $H$ whose logarithm is given by
\[
\log \left(\begin{matrix} 1 & \Z & \Z \\ 0 & 1 & \Z \\ 0&0&1\end{matrix}\right) = \left\{ \left(\begin{matrix} 0 & a & c \\ 0 & 0 & b \\ 0&0&0\end{matrix}\right) : a,b\in \Z,\, c \in \frac{1}{2}\Z,\, 2c=ab \text{ mod } 2 \right\}.
\]
This is \emph{not} an additive subgroup of $\mathfrak{h}$ since
\[
\left(\begin{matrix} 0 & 1 & 0 \\ 0 & 0 & 0 \\ 0&0&0\end{matrix}\right) + \left(\begin{matrix} 0 & 0 & 0 \\ 0 & 0 & 1 \\ 0&0&0\end{matrix}\right) = \left(\begin{matrix} 0 & 1 & 0 \\ 0 & 0 & 1 \\ 0&0&0\end{matrix}\right) \notin \log \left(\begin{matrix} 1 & \Z & \Z \\ 0 & 1 & \Z \\ 0&0&1\end{matrix}\right).
\]
 Similarly, the set of all elements of $\mathfrak{h}$ whose matrix entries are integers is a bracket-closed additive subgroup of $\mathfrak{h}$ whose exponential is not a subgroup of $H$.
\end{example}


We call a subgroup $H$ of $G$ \textbf{harmonious} if $\log H$ is an additive subgroup of $\mathfrak{g}$ that is bracket-closed in the sense that $[\log x, \log y]\in \log H$ for every $x,y\in H$. (This terminology is not standard.) Thus, harmonious subgroups of $G$ are those that most closely mimic the behaviour of the connected subgroups of $G$ under the exponential map. We call a lattice in $G$ that is also a harmonious subgroup of $G$ a \textbf{harmonious lattice} in $G$, noting that if $\Gamma$ is a harmonious lattice in $G$ then $\Lambda=\log \Gamma$ is a bracket-closed lattice in $\mathfrak{g}$.

\medskip

The remainder of this section is devoted to proving the following theorem, which states intuitively that subgroups of $G$ and bracket-closed additive subgroups of $\mathfrak{g}$ are, in some sense, ``equivalent up to bounded index''.  This result allows us to deduce various statements about lattices in simply connected nilpotent Lie groups from analogous statements for vector spaces (i.e., statements in the geometry of numbers), which are classical. Note that the theorem does not require $\Gamma$ to be discrete.
Recall that if $A \subseteq \mathfrak{g}$ and $\lambda \in \R$ then we define $\lambda \cdot A = \{\lambda a:a\in A\}$, so that if $\Lambda$ is an additive subgroup of $\mathfrak{g}$ then $(m\lambda) \cdot \Lambda \subseteq \lambda \cdot \Lambda$ for every $\lambda\in \R$ and $m\in \Z$. We also write $\mathscr{B}(A)$ for the smallest bracket-closed set containing $A$.

\begin{thm}
\label{thm:harmonious}
Let $G$ be a simply connected nilpotent Lie group of step $s$ with Lie algebra $\mathfrak{g}$. There exist  positive integers $C_1$ and $C_2$ depending only on $s$ such that if $\Gamma$ is a subgroup of $G$ then the sets
\[\mathscr{H}_-(\Gamma):= \exp\Bigr(C_1\cdot \spanZ(\log \Gamma)\Bigr)
\qquad\text{and}\qquad
\mathscr{H}_+(\Gamma):=\exp\left(C_1 \cdot \mathscr{B}\left(\frac{1}{C_1} \cdot \spanZ (\log \Gamma)\right)\right) \]
are harmonious subgroups of $G$ such that 
\[C_1 \cdot \log \Gamma \subseteq \log \mathscr{H}_-(\Gamma) \subseteq \log \Gamma \subseteq \log \mathscr{H}_+(\Gamma) \subseteq \frac{1}{C_2} \cdot \log \Gamma.\]
\end{thm}

In \cref{cor:harmonious_bounded_index} we show moreover that if $\Gamma$ is discrete then the harmonious subgroups $\mathscr{H}_-(\Gamma)$ and $\mathscr{H}_+(\Gamma)$ have index 
 bounded by a constant depending only on the step and dimension of $G$.

\begin{example}
We continue to analyze the Heisenberg group as studied in \Cref{ex:heisenberg_def,ex:heisenberg_subgroups}. Although the subgroup $\Gamma$ of $H$ consisting of those elements of $H$ with integer matrix entries is not harmonious, we can write
\[
  \left(\begin{matrix} 1 & 2\Z & \Z \\ 0 & 1 & 2\Z \\ 0&0&1\end{matrix}\right) \subseteq \Gamma = \left(\begin{matrix} 1 & \Z & \Z \\ 0 & 1 & \Z \\ 0&0&1\end{matrix}\right) \subseteq \left(\begin{matrix} 1 & \Z & \frac{1}{2}\Z \\ 0 & 1 & \Z \\ 0&0&1\end{matrix}\right)
\]
with the two outer two lattices being harmonious subgroups of $H$. Moreover, these subgroups arise naturally from the original subgroup $\Gamma$ as
\begin{align*}
\left(\begin{matrix} 1 & 2\Z & \Z \\ 0 & 1 & 2\Z \\ 0&0&1\end{matrix}\right) = \exp\left(\operatorname{span}_\Z (2\cdot \log \Gamma) \right)
\text{and}\qquad 
\left(\begin{matrix} 1 & \Z & \frac{1}{2}\Z \\ 0 & 1 & \Z \\ 0&0&1\end{matrix}\right) = \exp\left(\operatorname{span}_\Z (\log \Gamma)\right).
\end{align*}
Similar remarks apply to the bracket-closed additive lattice in $\mathfrak{h}$ consisting of those elements of $\mathfrak{h}$ with integer matrix entries.
\end{example}

The methods used to prove \cref{thm:harmonious} are based on those of Breuillard and Green \cite{breuillard2011approximate}. In particular, we will make use of the following lemma of Tessera and Tointon \cite{MR3877012} that is also proved using the methods of \cite{breuillard2011approximate}.

\begin{lemma}
\label{lem:additive_to_harmonious}
Let $G$ be a simply connected nilpotent Lie group of step $s$ with Lie algebra $\mathfrak{g}$. There exists an integer constant $C=C(s)$ such that if $\Lambda$ is a bracket-closed additive subgroup of the Lie algebra $\mathfrak{g}$ then $\exp ( C \cdot \Lambda)$ is a harmonious subgroup of $G$.
\end{lemma}

\begin{proof}
This is essentially \cite[Lemma 4.3]{MR3877012}. It is not stated that $C\cdot \Lambda$ is bracket-closed, but this is obvious since $[C \cdot\Lambda,C\cdot\Lambda]= C^2 \cdot [\Lambda,\Lambda] \subseteq C\cdot \Lambda$.
\end{proof}

\begin{remark}\label{rmk:C_divides_C1}
The constant $C_1$ appearing in \cref{thm:harmonious} will be taken to be a multiple of the constant $C$ appearing in \cref{lem:additive_to_harmonious}. This will be important in the proof of \cref{prop:exploring_Carnot_subgroup}.
\end{remark}

We begin by stating the following lemma of Lazard \cite{lazard1954groupes} as presented in \cite[Lemmas 5.2 and 5.3]{breuillard2011approximate}. This lemma was first applied to the structure theory of approximate groups in the work of Fisher, Katz, and Peng \cite{fisher2009freiman}. Given a simply connected nilpotent Lie group $G$, we define $x^\alpha=\exp(\alpha \log x)$ for every $x\in G$ and $\alpha \in \R$.

\begin{lemma}[Lazard]
Let $G$ be a simply connected nilpotent Lie group of step $s$ with Lie algebra $\mathfrak{g}$. There exists $\ell\geq 1$ and sequences of rational numbers $\alpha_1,\ldots,\alpha_\ell$, $\beta_1,\ldots,\beta_\ell$, $\gamma_1,\ldots,\gamma_\ell$, and $\delta_1,\ldots,\delta_\ell$ depending only on $s$ such that
\[
\exp(\log x+\log y)=x^{\alpha_1}y^{\beta_2}\cdots x^{\alpha_\ell} y^{\beta_\ell}
\]
and
\[
\exp([\log x,\log y])=x^{\gamma_1}y^{\delta_1}\cdots x^{\gamma_\ell} y^{\delta_\ell}
\]
for every $x,y \in G$.
\end{lemma}

\begin{corollary}[Expansion of sums]
\label{cor:expansion_of_sums}
Let $G$ be a simply connected nilpotent Lie group of step $s$ with Lie algebra $\mathfrak{g}$. For each $n\geq 2$ there exists $\ell\geq 1$, a sequence of rational numbers $\alpha_1,\ldots,\alpha_\ell$, and a sequence of indices $i_1,\ldots,i_\ell \in \{1,\ldots,n\}$, all depending only on $s$ and $n$, such that
\[
\exp(\log x_1+\log x_2 + \cdots + \log x_n )=x^{\alpha_1}_{i_1} x^{\alpha_2}_{i_2}\cdots x^{\alpha_\ell}_{i_\ell}
\]
for every $x_1,\ldots,x_n \in G$.
\end{corollary}

\begin{corollary}
\label{cor:expansion_of_sums2}
Let $G$ be a simply connected nilpotent Lie group of step $s$ with Lie algebra $\mathfrak{g}$, let $\Gamma$ be a subgroup of $G$ and let $\Lambda = \log \Gamma$. For each $n\geq 1$, there exists a natural number $C=C(s,n)$ such that if $X_1,\ldots,X_n \in C\cdot \Lambda$ then $X_1+\cdots+X_n \in \Lambda$.
\end{corollary}

\begin{proof}
Let $C=C(s,n)$ be the least common multiple of the denominators of the numbers $\alpha_1,\ldots,\alpha_\ell$ appearing in \cref{cor:expansion_of_sums} when written in reduced form.
\end{proof}

\begin{lemma}
\label{lem:multlilinear_monomials}
Let $G$ be a simply connected nilpotent Lie group of step $s$ with Lie algebra $\mathfrak{g}$, let $\Gamma$ be a subgroup of $G$ and let $\Lambda = \log \Gamma$. There exists a natural number $C=C(s)$ such that
\[
\Bigl\{L(X_1,\ldots,X_n) : X_1,\ldots,X_n \in C^{n-1} \cdot \Lambda \Bigr\} \subseteq \Lambda.
\]
for every multilinear Lie monomial $L$. 
\end{lemma}

(A Lie monomial $L(X_1,\ldots,X_n)$ is multilinear when each variable appears at most once.)

\begin{proof}
Let $C=C(s)$ be the least common multiple of the denominators of the numbers $\gamma_1,\ldots,\gamma_\ell$ and $\delta_1,\ldots,\delta_\ell$ appearing in Lazard's lemma when written in reduced form. We will prove the claim by induction on $n$, the case $n=1$ being vacuous. If $n>1$ then 
\[
L(X_1,\ldots,X_n)=[L_1(X_{\pi(1)},\ldots,X_{\pi(j)}),L_2(X_{\pi(j+1)},\ldots,X_{\pi(n)})] 
\]
for some $1\leq j < n$, some multilinear Lie monomials $L_1,L_2$ and some permutation $\pi:\{1,\ldots,n\}\to\{1,\ldots,n\}$. We may assume without loss of generality that $\pi$ is the identity permutation. By the induction hypothesis, if $X_1,\ldots,X_n\in C^{n-2} \cdot \Lambda$ then $L_1(X_1,\ldots,X_j),L_2(X_{j+1},\ldots,X_n)\in \Lambda$. As such, if $X_1,\ldots,X_n \in C^{n-1} \cdot \Lambda$ then we have by multlilinearity that
\[
L_1(X_1,\ldots,X_j)=C^jL_1(C^{-1}X_1,\ldots,C^{-1}X_j)\in C^{j} \cdot \Lambda \subseteq C \cdot \Lambda \] and  \[L_2(X_{j+1},\ldots,X_n)=C^{n-j}L_2(C^{-1}X_{j+1},\ldots,C^{-1}X_n)\in C^{n-j} \cdot \Lambda \subseteq C \cdot \Lambda.
\]
On the other hand, if $X,Y\in C\cdot \Lambda$ then $e^{\gamma_j X},e^{\delta_j Y} \in \Gamma$ for every $1\leq j \leq \ell$, so that
\[
\exp([X,Y])=e^{\gamma_1 X} e^{\delta_1 Y} \cdots e^{\gamma_\ell X} e^{ \delta_\ell Y} \in \Gamma
\]
and hence that $[X,Y] \in \Lambda$. It follows that if $X_1,\ldots,X_n \in C^{n-1} \cdot \Lambda$ then $L(X_1,\ldots,X_n)\in \Lambda$ as claimed.
\end{proof}

\begin{corollary}
\label{cor:multlilinear_monomials}
Let $G$ be a simply connected nilpotent Lie group of step $s$ with Lie algebra $\mathfrak{g}$, let $\Gamma$ be a subgroup of $G$ and let $\Lambda = \log \Gamma$. There exists a natural number $C=C(s)$ such that
\[
\Bigl\{L(X_1,\ldots,X_n) : X_1,\ldots,X_n \in \Lambda,\; \text{ $X_i\in C^{d(d-1)}\cdot \Lambda$ for some $1\leq i \leq n$} \Bigr\} \subseteq \Lambda.
\]
for every Lie monomial $L$ of degree $d$ that depends on every variable. 
\end{corollary}

\begin{proof}
It suffices without loss of generality to consider the case that $L$ is multilinear.
Let $C$ be the constant from \cref{lem:multlilinear_monomials}. Let $X_1,\ldots,X_n \in \Lambda$, and let $L(X_1,\ldots,X_n)$ be a multlilinear Lie monomial depending on every variable. Such a Lie monomial necessarily has degree $d=n$. For each $1\leq i \leq n$ we can write
\[
L(X_1,\ldots,X_n) = L(C^{(n-1)}X_1,\ldots,C^{(n-1)}X_{i-1},C^{-(n-1)^2}X_i,C^{(n-1)}X_{i+1},\ldots,C^{(n-1)}X_n).
\]
If $X_i \in C^{n(n-1)}\cdot \Lambda=C^{(n-1)^2+(n-1)} \cdot \Lambda$ then $C^{-(n-1)^2}X_i \in C^{(n-1)}\cdot \Lambda$, so that the claim follows from \cref{lem:multlilinear_monomials}.
\end{proof}

\begin{lemma}
\label{lem:good_sublattice}
Let $G$ be a simply connected nilpotent Lie group of step $s$ with Lie algebra $\mathfrak{g}$. There exists a constant $C$ depending only on $s$ such that if $\Gamma$ is a subgroup of $G$ and we write $\Lambda = \log \Gamma$ then $C\cdot \Lambda$ is bracket-closed and $X+Y \in \Lambda$ for every $X\in C\cdot \Lambda$ and $Y\in \Lambda$. 
\end{lemma}

\begin{proof}
We recall the Zassenhaus formula, which states for simply connected nilpotent Lie groups that
\[
\exp(X+Y) = e^X e^Y e^{L_2(X,Y)}e^{L_3(X,Y)}\cdots e^{L_s(X,Y)}
\]
for every $X,Y\in \mathfrak{g}$, where, for each $i\geq 1$, $L_i$ is a homogeneous Lie polynomial of degree $i$
of the form
\[
L_i(X,Y) = \sum_{j=1}^{r_i} a_{i,j} L_{i,j}(X,Y)
\]
where $a_{i,j}$ are rational numbers and $L_{i,j}$ are Lie monomials depending on both variables.
%
%
Let $C_1$ be the least common multiple of the denominators of the rational numbers $a_{i,j}$ and let $C_2$ be the least common multiple of the constants $C(s,2),C(s,3),\ldots,C(2,\max_i r_i)$ appearing in \cref{cor:expansion_of_sums2}. By \cref{cor:multlilinear_monomials}, there exists a constant $C_3=C_3(s)$ such that if $X\in C_3 \cdot \Lambda$ and $Y\in \Lambda$ then $L_{i,j}(X,Y)\in \Lambda$ for every $1\leq i \leq s$ and $1\leq j \leq r_i$. Let $C=C_1C_2C_3$. Since $L_{i,j}$ depends on $X$, it follows that if $X \in C \cdot \Lambda$ then $L_{i,j}(X,Y)=(C_1C_2)^{d_{i,j}}L_{i,j}((C_1C_2)^{-1}X,Y) \in (C_1C_2)\cdot \Lambda$ for every $y\in \Lambda$, where $d_{i,j}$ is the degree of $X$ in $L_{i,j}$. It follows in particular that if $X\in C\cdot \Lambda$ then $a_{i,j} L_{i,j}(X,Y) \in C_2 \cdot \Lambda$ for every $1\leq i \leq s$ and $1\leq j \leq r_i$ and hence by \cref{cor:expansion_of_sums2} that $L_i(X,Y) \in \Lambda$ for every $X\in C \cdot \Lambda$ and $Y\in \Lambda$. Since $\Gamma$ is a subgroup of $G$, it follows by the Zassenhaus formula that $X+Y\in \Lambda$ for every $X\in C\cdot \Lambda$ and $Y\in \Lambda$ as claimed. Moreover, if $X,Y\in C\cdot \Lambda$ then $[X,Y]=C[X,C^{-1}Y] \in C\cdot \Lambda$ since $[X',Y']\in \Lambda$ for every $X'\in C\cdot \Lambda$ and $Y'\in \Lambda$, so that $C\cdot \Lambda$ is bracket-closed as claimed. 
\end{proof}

We are now ready to prove \cref{thm:harmonious}.

\begin{proof}[Proof of \cref{thm:harmonious}]
We begin by proving the claim concerning $\mathscr{H}_-(\Gamma)$.
 Let $C_{-1}=C_{-1}(s)$ be the constant from \cref{lem:good_sublattice}, let $C_0=C_0(s)$ be the constant from \cref{lem:additive_to_harmonious}, and let $C_1=C_{-1}C_0$.
 Let $G$ be a simply connected nilpotent Lie group of step $s$, let $\mathfrak{g}$ be the Lie algebra of $G$, let $\Gamma$ be a subgroup of $G$ and let $\Lambda = \log \Gamma$.
  \cref{lem:good_sublattice} implies that $C_1 \cdot \Lambda$ is bracket-closed and that $C_{-1}\cdot \Lambda + \Lambda \subseteq \Lambda$, and it follows by induction on the number of terms in a linear combination that $\spanZ(C_{-1}\cdot \Lambda)=C_{-1}\cdot \spanZ(\Lambda)$ is contained in $\Lambda$. 
Since the $\Z$-span of a bracket-closed set is bracket-closed, it follows that $\spanZ(C_{-1} \cdot \Lambda)$ is a bracket-closed additive subgroup of $\mathfrak{g}$ and hence by \cref{lem:additive_to_harmonious} that $\spanZ (C_1 \cdot \Lambda)=C_0\cdot \spanZ(C_{-1}\cdot \Lambda)$ is a bracket-closed additive subgroup of $\mathfrak{g}$ whose exponential $\mathscr{H}_-(\Gamma)$ is a harmonious subgroup of $G$ satisfying the required set inclusion $C_1 \cdot \log \Gamma \subseteq \log \mathscr{H}_-(\Gamma) \subseteq \log \Gamma$.

\medskip

Now consider the set $\mathscr{H}_+(\Gamma)$ defined by
\[\mathscr{H}_+(\Gamma):=\exp\left(C_1 \cdot \mathscr{B}\left(\frac{1}{C_1} \cdot \spanZ (\Lambda)\right)\right)=\exp\left(C_1 \cdot \spanZ \mathscr{B}\left(\frac{1}{C_1} \cdot \Lambda\right)\right).\]
 Since $\spanZ \mathscr{B}\left(\frac{1}{C_1} \cdot \Lambda\right)$ is a bracket-closed additive subgroup of $\mathfrak{g}$ and $C_0$ divides $C_1$, \cref{lem:additive_to_harmonious} implies that $\mathscr{H}_+(\Gamma)$ is a harmonious subgroup of $G$. Moreover, $\log \mathscr{H}_+(\Gamma)$ trivially contains $\spanZ \Lambda$. As such, it remains only to prove that there exists a constant $C_2=C_2(s)$ such that $\log \mathscr{H}_+(\Gamma) \subseteq \frac{1}{C_2} \cdot \Lambda$. Since $C_1\cdot \spanZ (\Lambda)$ is bracket-closed
 and $[\mathfrak{g}_i,\mathfrak{g}_j]\subseteq \mathfrak{g}_{i+j}$ we have that
 \[
\left[\left(\frac{1}{n}\cdot \spanZ(\Lambda)\right)\cap \mathfrak{g}_i, \left(\frac{1}{m}\cdot \spanZ(\Lambda)\right)\cap \mathfrak{g}_j\right] \subseteq \left(\frac{1}{C_1^2 nm} \Lambda\right) \cap \mathfrak{g}_{i+j}
 \]
 for every pair of integers $n,m \geq 0$ and $1 \leq i, j \leq s$. Thus, it follows by induction on $i$ that
 \[
\mathscr{B}\left(\frac{1}{C_1} \cdot \spanZ (\Lambda)\right) \cap \mathfrak{g}_i \subseteq \frac{1}{C_1^{3i-2}} \cdot (\Lambda \cap \mathfrak{g}_i)
 \]
for every $i\geq 1$ and hence that
 \[
\mathscr{B}\left(\frac{1}{C_1} \cdot \spanZ (\Lambda)\right) \subseteq \frac{1}{C_1^{3s-2}} \cdot \Lambda.
 \]
 This implies that the claim holds with $C_2=C_1^{3(s-1)}$.
\end{proof}



\subsection{Comparing additive and multiplicative indices}

In this section we prove bounds on the index $[\mathscr{H}_+(\Gamma):\mathscr{H}_-(\Gamma)]$. We will deduce these bounds from the following proposition, which lets us compare additive and multiplicative indices in the Lie algebra of a simply connected nilpotent Lie group.

\begin{proposition}[Index sandwich]
\label{prop:harmonious_indices}
Let $\Gamma_1 \subseteq \Gamma_2$ be lattices in a simply connected nilpotent Lie group $G$ with Lie algebra $\mathfrak{g}$, and suppose that $\Lambda_1\subseteq \Lambda_2$ are additive lattices in $\mathfrak{g}$.
\begin{enumerate}
\item If $\Lambda_1 \subseteq \log \Gamma_1 \subseteq \log \Gamma_2 \subseteq \Lambda_2$ then $[\Gamma_2:\Gamma_1] \leq [\Lambda_2:\Lambda_1]$. 
\item If  $\log \Gamma_1 \subseteq \Lambda_1  \subseteq \Lambda_2 \subseteq \log \Gamma_2$ then $[\Gamma_2:\Gamma_1] \geq [\Lambda_2:\Lambda_1]$. 
\end{enumerate}
In particular, if $\Gamma_1$ and $\Gamma_2$ are harmonious in $G$ then $[\Gamma_2:\Gamma_1]=[\log \Gamma_2:\log \Gamma_1]$.
\end{proposition}

The proof of this proposition will require the following classical fact.

\begin{proposition}[Compatibility of Haar measures]
\label{prop:Haar_to_Lebesgue}
Let $G$ be a simply connected nilpotent Lie group and let $\mathfrak{g}$ be the Lie algebra of $G$. If $\mu$ is a translation-invariant, locally finite measure on~$\mathfrak{g}$ (i.e., a Lebesgue measure) then the pushforward of $\mu$ by the exponential map is a locally finite measure on $G$ that is both left and right invariant (i.e., a bi-invariant Haar measure).
\end{proposition}

\begin{proof}[Proof of \cref{prop:Haar_to_Lebesgue}]
It suffices to prove that for every $X\in \mathfrak{g}$, the maps $\ell_X:Y \mapsto X \diamond Y$ and $r_X:Y \mapsto Y \diamond X$ preserve the Lebesgue measure on $\mathfrak{g}$. For this, it suffices to prove that the total derivatives $D\ell_X$ and $Dr_X$ have $|\det(D\ell_X)|=|\det(Dr_X)|=1$ at every $Y\in \mathfrak{g}$. To prove this, it suffices to prove that both $D\ell_X$ and $Dr_X$ can be expressed as the sum of the identity and a nilpotent linear transformation, which can be done via an explicit computation with the BCH formula. For details see e.g.\ \cite[Proposition 2.1.1]{Wang2023}.
\end{proof}

\begin{proof}[Proof of \cref{prop:harmonious_indices}]
%
%
We start by constructing large sets in $\mathfrak{g}$ that have ``small boundary-to-volume ratio'' in both the additive and multiplicative senses. (That is, the sets we construct will yield a F{\o}lner sequence for both addition and BCH multiplication on $\mathfrak{g}$.)  If $G$ were assumed to be a Carnot group we could use the logarithms of balls in a homogeneous left-invariant metric; we will perform a similar construction for a general Lie group. Let $(\mathfrak{g}_i)_{i\geq 0}$ be the lower central series of $\mathfrak{g}$, and for each $i\geq 0$ let $V_i$ be such that $\mathfrak{g}_i = V_i \oplus \mathfrak{g}_{i+1}$, so that we can write $\mathfrak{g}=V_1 \oplus V_2 \oplus \cdots \oplus V_s$. For notational convenience we also write $V_i=\{0\}$ for $i>s$. In contrast to the Carnot case, it is not necessarily the case that $V_1$ generates $\mathfrak{g}$ as a Lie algebra or that $[V_i,V_j]\subseteq V_{i+j}$, but we do have that $[V_{i},V_j] \subseteq [\mathfrak{g}_i,\mathfrak{g}_j] \subseteq \mathfrak{g}_{i+j} =  \bigoplus_{k = i+j}^sV_{k}$ for every $i,j \geq 1$.
Fix an isomorphism of vector spaces $\mathfrak{g} \cong \R^d$ for some $d\geq 1$ and let $\|\cdot \|_\infty$ be the associated $\infty$-norm on $\mathfrak{g}$. For each $\lambda>0$ let
\[
F_\lambda := \{X\in \mathfrak{g}: \max_i \lambda^{-i}\|X_i\|_\infty \leq 1\}=\{X\in \mathfrak{g}: \max_i \|X_i\|_\infty^{1/i} \leq \lambda\},
\]
where we write $X=\sum_i X_i$ for the decomposition of $X$ induced by the decomposition $\mathfrak{g}=V_1 \oplus V_2 \oplus \cdots \oplus V_s$. The volume of $F_\lambda$ satisfies
\[\Vol (F_\lambda)=\lambda^q\] for an integer $q\geq 1$, which can be expressed as $q=\sum_{i=1}^\ell i \dim(V_i)$ (this is equal to the \emph{homogeneous dimension} of $G$). For each $X\in \mathfrak{g}$ we define
\[
\pnorm{X}=\inf \{\lambda>0: X\in F_\lambda\} = \max_i \|X_i\|^{1/i},
\]
so that $\pnorm{X} \leq \|X\|_\infty^{1/i}$ for every $1\leq i \leq s$ and $X \in \mathfrak{g}_i$ with equality if $X\in V_i$.
Note that this is \emph{not} a norm on $\mathfrak{g}$ since it does not satisfy $\pnorm{\lambda X}=\lambda \pnorm{X}$ for all $X\in \mathfrak{g}$ and $\lambda>0$. (Rather, it scales under a certain graded dilation map as in the Carnot case.) Moreover, unlike the metrics we considered on Carnot groups, $\pnorm{\cdot}$ will not be left-invariant in general. Nevertheless, it does trivially satisfy the triangle inequality in the form
\begin{multline}\label{eq:pnorm_triangle}
\pnorm{X+Y} = \max_i \|X_i+Y_i\|_\infty^{1/i} \leq \max_i (\|X_i\|_\infty+\|Y_i\|_\infty)^{1/i}  
\\\leq \max_i (\|X_i\|_\infty^{1/i}+\|Y_i\|_\infty^{1/i}) \leq \pnorm{X}+\pnorm{Y}.
\end{multline}
As with norms, writing $\pnorm{X} = \pnorm{(X+Y)-Y}$ yields the reverse inequality
\begin{equation}\label{eq:pnorm_reverse_triangle}\pnorm{X+Y} \geq \pnorm{X}-\pnorm{Y},\end{equation}
so that
$F_{\lambda-\pnorm{Y}} \subseteq F_\lambda + Y \subseteq F_{\lambda+\pnorm{Y}}$ for every $Y \in \mathfrak{g}$ and $\lambda \geq \pnorm{Y}$.

\medskip

We will need similar inequalities for the BCH product $\pnorm{X\diamond Y}$. Since the Lie bracket is bilinear and $\mathfrak{g}$ is finite-dimensional, there exists a constant $C_1$ such that $\|[X,Y]\|_\infty \leq C_1\|X\|_\infty \|Y\|_\infty$ for every $X,Y\in \mathfrak{g}$. This implies that there exists a constant $C_2$ such that
%
%
\[
\pnorm{[X_i,Y_j]} \leq \|[X_i,Y_j]\|_\infty^{1/(i+j)} \leq C_1^{1/(i+j)}\|X_i\|_\infty^{1/(i+j)}\|Y_j\|_\infty^{1/(i+j)} \leq C_2 \pnorm{X_i}^{i/(i+j)}\pnorm{Y_j}^{j/(i+j)}
\]
for every $X_i\in V_i$ and $Y_j \in V_j$. Together with \eqref{eq:pnorm_triangle} this implies that there exists a constant $C_3$ such that
\[
\pnorm{[X,Y]} \leq C_3 \max\left\{\pnorm{X}^{1-\theta}\pnorm{Y}^{\theta}:\frac{1}{s}\leq \theta \leq \frac{s-1}{s}\right\}
\]
for every $X,Y\in \mathfrak{g}$. It follows by induction on $k\geq 2$ that if $L(X,Y)$ is any Lie monomial of degree $k \geq 2$ then, writing $L(X,Y)=[L_1(X,Y),L_2(X,Y)]$ for two Lie monomials of degree $a,k-a<k$,
\begin{align*}
\pnorm{L(X,Y)} &\leq 
C_3 \max\left\{\pnorm{L_1(X,Y)}^{1-\theta}\pnorm{L_2(X,Y)}^{\theta}: \frac{1}{s} \leq \theta \leq \frac{s-1}{s}\right\}
\\&\leq
C^{k-1}_3 \max\left\{\pnorm{X}^{1-\theta}\pnorm{Y}^{\theta}:\frac{1}{s^k} \leq \theta \leq \frac{s^k-1}{s^k}\right\}
\end{align*}
for every $X,Y\in \mathfrak{g}$. (The case that one of the monomials $L_1$ or $L_2$ has degree \emph{one}, and so is equal to $X$ or $Y$, must be checked separately.) We deduce from this together with \eqref{eq:pnorm_triangle} and the definition of BCH multiplication that there exists a constant $C_4$ such that
\begin{equation}
\label{eq:pnorm_BCH_triangle}
\pnorm{X\diamond Y} \leq \pnorm{X}+\pnorm{Y} + C_4\max\Bigl\{\pnorm{X}^{1-\theta}\pnorm{Y}^{\theta}: \frac{1}{s^s} \leq \theta \leq \frac{s^s-1}{s^s} \Bigr\}
\end{equation}
for every $X,Y\in \mathfrak{g}$.  Since $-Y$ is both the additive and BCH inverse of $Y$, we can write $X=X \diamond Y \diamond (-Y)$ to obtain that there exists a constant $C_5$ such that the complementary inequality 
\begin{align}
\pnorm{X} &\leq \pnorm{X\diamond Y}+ \pnorm{Y} + C_4\max\Bigl\{\pnorm{X \diamond Y}^{1-\theta}\pnorm{Y}^{\theta}: \frac{1}{s^s} \leq \theta \leq \frac{s^s-1}{s^s} \Bigr\}
\nonumber
\\
&\leq \pnorm{X\diamond Y}+ C_5\pnorm{Y} + C_5 \max\Bigl\{\pnorm{X}^{1-\theta}\pnorm{Y}^{\theta}: \frac{1}{s^{2s}} \leq \theta \leq \frac{s^{2s}-1}{s^{2s}} \Bigr\}
\label{eq:pnorm_BCH_reverse_triangle}
\end{align}
holds for every $X,Y\in \mathfrak{g}$. 

\medskip

We now use the sets $F_\lambda$ to prove the claim about indices. Suppose that $\Lambda$ is an additive lattice in $\mathfrak{g}$ and that $\Gamma$ is a lattice in $G$. Let $K_\Lambda$ be a fundamental domain for $\Lambda$ in $\mathfrak{g}$ and let $K_\Gamma$ be a fundamental domain for $\Gamma$ in $G$. 
 Since $K_\Lambda \cup \log K_\Gamma$ is compact, $\max\{\pnorm{Y}:Y\in K_\Lambda \cup \log K_\Gamma\}$ is finite. As such, it follows from \eqref{eq:pnorm_triangle} and \eqref{eq:pnorm_reverse_triangle} that there exist positive constants $C$ and $\eps$ (depending on $K_\Lambda$ and $K_\Gamma$) such that
\[
\pnorm{X}-C \leq \pnorm{X+Y} \leq \pnorm{X} + C
 \]
 and
 \[
 \pnorm{X}-C\pnorm{X}^
{1-\eps} -C \leq \pnorm{X 
\diamond Y} \leq \pnorm{X} + C \pnorm{X}^
{1-\eps}+C
\]
for every $x\in \mathfrak{g}$ and $y\in K_\Lambda \cup \log K_\Gamma$. %
Since the sets $(X+K_\Lambda : X \in \Lambda)$ and $(X \diamond \log K_\Gamma : X \in \log \Gamma)$ each cover $\mathfrak{g}$ (with distinct sets having measure-zero intersection), it follows that 
\[
 F_{\lambda - C} \subseteq \bigcup_{X \in F_\lambda \cap \Lambda} (X+K_\Lambda) \subseteq F_{\lambda + C}
\]
and
\[
 F_{\lambda-C\lambda^{1-\eps}-C} \subseteq \bigcup_{X \in F_\lambda \cap \log \Gamma} (X \diamond \log K_\Gamma) \subseteq F_{\lambda + C\lambda^{1-\eps}+C}
\]
for every $\lambda \geq 1$ such that $\lambda-C\lambda^{1-\eps}-C>0$. Taking volumes and using that addition and BCH multiplication are both measure-preserving, we deduce that
\begin{equation*}
(\lambda-C)^q \leq |F_\lambda \cap \Lambda| \cdot \Vol (K_\Lambda) = \Vol \left(\bigcup_{X\in F_\lambda \cap \Lambda} (X+K_\Lambda) \right) \leq (\lambda+C)^q
\label{eq:Folner_Lambda_final}
\end{equation*}
and
\begin{equation*}
(\lambda - C\lambda^{1-\eps}-C)^q
 \leq |F_\lambda \cap \log \Gamma| \cdot \Vol (\log K_\Gamma) = \Vol \left(\bigcup_{X\in F_\lambda \cap \log \Gamma} (X \diamond \log K_\Gamma) \right)  \leq (\lambda + C\lambda^{1-\eps}+C)^q
\label{eq:Folner_Gamma_final}
\end{equation*}
for every $\lambda \geq 1$ such that $\lambda-C\lambda^{1-\eps}-C>0$, so that
\begin{equation}
\Vol (K_\Lambda) = \lim_{\lambda \to \infty} \frac{\lambda^q}{|F_\lambda \cap \Lambda|} \qquad \text{ and } \qquad \Vol (\log K_\Gamma) = \lim_{\lambda \to \infty} \frac{\lambda^q}{|F_\lambda \cap \log \Gamma|}.
\end{equation}
This is easily seen to imply the claim.
\end{proof}

\begin{corollary}
\label{cor:harmonious_bounded_index}
Let $G$ be a simply connected nilpotent Lie group of step $s$ and dimension $d$, and let the constants $C_1=C_1(s)$ and $C_2=C_2(s)$ be as in \cref{thm:harmonious}. If $\Gamma$ is a discrete subgroup of $G$ then the harmonious subgroups $\mathscr{H}_+(\Gamma)$ and $\mathscr{H}_-(\Gamma)$ satisfy the index bounds
\[
[\mathscr{H}_+(\Gamma):\Gamma][\Gamma:\mathscr{H}_-(\Gamma)]= [\mathscr{H}_+(\Gamma):\mathscr{H}_-(\Gamma)] \leq (C_2C_1)^d.
\]
\end{corollary}

\begin{proof} 
We may assume without loss of generality that $\Gamma$ is a lattice, replacing $G$ by $\mathscr{C}(\Gamma)$  otherwise. Since $\log \mathscr{H}_+(\Gamma) \subseteq C_1C_2 \cdot \log \mathscr{H}_-(\Gamma)$ we have that $[\log \mathscr{H}_+(\Gamma):\log \mathscr{H}_-(\Gamma)]\leq (C_1C_2)^d$, and the claim follows from \cref{prop:harmonious_indices}.
\end{proof}

\section{Proof of the main theorems}
\label{sec:wrap_up}

In this section we prove our main theorems, \cref{thm:main,thm:main_generalized}. We begin by proving the following proposition, which is a direct analogue of \cref{thm:Abelian_convex} for Carnot groups. 

\begin{proposition}
\label{prop:exploring_Carnot_subgroup}
Let $G$ be a Carnot group with Lie algebra $\mathfrak{g}$, let $d_G$ be a left-invariant homogeneous metric on $G$, and for each $r>0$ let $B_r$ be the ball of radius $r$ around $\mathrm{id}$ in $(G,d_G)$. Then
\[
\sup\left\{\#\{\langle H \cap B_{2^k} \rangle : k \in \Z\} : H \text{ a discrete subgroup of $G$}\right\} <\infty.
\]
\end{proposition}

It will suffice for our applications that all relevant constants depend on the pair $(G,d_G)$ in an arbitrary fashion. 

\begin{proof}[Proof of \cref{prop:exploring_Carnot_subgroup}]
Since $d_G$ is consistent with the usual topology of $G$, the (closed) ball $B_r$ is a compact subset of $G$ containing a neighbourhood of the identity for each $r>0$. In particular, there exists a convex, symmetric subset $K$ of $\mathfrak{g}$ and a constant $C_0 \geq 1$ such that $K \subseteq \log B_1 \subseteq C_0 K \subseteq \delta_{C_0}(K)$, where $(\delta_\lambda)_{\lambda>0}$ is the dilation semigroup on the stratified Lie algebra $\mathfrak{g}$. Thus, the balls $B_r$ are sandwiched between the exponentials of the dilates of $K$:
\begin{equation}
\label{eq:convex_sandwich}
\delta_r(K) \subseteq \log B_r \subseteq \delta_{C_0r}(K)
\end{equation}
for every $r>0$.
The sets $\delta_\lambda(K)$ are all convex and symmetric since they are linear images of the convex symmetric set $K$; this will allow us to apply \cref{thm:Abelian_convex} to an appropriately chosen additive subgroup of $\mathfrak{g}$.

\medskip



Let $\Lambda = \log H$, and for each $k\in \Z$ let $H_k = \langle H \cap B_{2^k} \rangle$ and $\Lambda_k = \log H_k$. By \cref{thm:harmonious}, there exists an integer constant $C_1$ such that $\tilde \Lambda := \operatorname{span}_\Z(C_1\cdot \Lambda)$ is an additive, bracket-closed lattice in $\mathfrak{g}$ whose exponential $\mathscr{H}_-(H)$ is a harmonious subgroup of $G$ that is contained in $H$.
   It follows by a direct application of \cref{thm:Abelian_convex} that there exists a constant $C_2$ such that
\[
\#\{ \operatorname{span}_\Z (\tilde \Lambda \cap \delta_\lambda(K)) : \lambda >0 \} \leq C_2.
\]
For each $k\geq 0$ we have trivially that $H_k \cap B_{2^k} = H \cap B_{2^k}$ and hence that $\Lambda_k \cap \delta_{2^k}(K)=\Lambda \cap \delta_{2^k}(K)$. In particular, 
\[\tilde \Lambda \cap \delta_{2^k}(K) \subseteq \tilde \Lambda \cap B_{2^k} \subseteq  \Lambda \cap B_{2^k} = \Lambda_k \cap B_{2^k}.\]
On the other hand, letting $m=\lceil \log_2 C_0 \rceil$, we also have by \eqref{eq:convex_sandwich} that
\[
\tilde \Lambda \cap \delta_{2^{k+m}}(K) \supseteq \tilde \Lambda \cap B_{2^k} \supseteq (C_1\cdot \Lambda) \cap B_{2^k} \supseteq C_1\cdot (\Lambda \cap B_{2^k}) = C_1\cdot (\Lambda_k \cap B_{2^k}).
\]
Thus, if we define $\tilde \Lambda_k = \operatorname{span}_\Z (\tilde \Lambda \cap \delta_{2^k}(K))$ for each $k>0$ then
\begin{equation}
\label{eq:lambda_tilde_in_the_middle}
C_1\cdot \operatorname{span}_\Z(\Lambda_{k-m} \cap B_{2^{k-m}}) \subseteq \tilde \Lambda_k \subseteq \operatorname{span}_\Z(\Lambda_k \cap B_{2^k})
\end{equation}
for every $k\in \Z$. 

\medskip

Consider an interval $[a,b]\cap \Z$ such that $\tilde \Lambda_k$ does not change as $k$ varies over $[a,b]\cap \Z$. 
Then we have by \eqref{eq:lambda_tilde_in_the_middle} that
\[
C_1\cdot \spanZ(\Lambda_{b-m} \cap B_{2^{b-m}}) \subseteq \spanZ(\Lambda_a\cap B_{2^a})
\]
and hence that
\[
[\spanZ(\Lambda_{b-m} \cap B_{2^{b-m}}):\spanZ(\Lambda_a\cap B_{2^a})]\leq C_1^d.
\]
Since strict sublattices have index at least $2$, this in turn implies that 
\[
\#\{\spanZ(\Lambda_{k} \cap B_{2^{k}}) : k \in [a,b]\} \leq m + \#\{\spanZ(\Lambda_{k} \cap B_{2^{k}}) : k \in [a,b-m]\} \leq m + d \log_2 C_1.
\]
Since $[a,b]$ was an arbitrary interval over which $\tilde \Lambda_k$ remained constant, it follows that there exist constants $C_3$ and $C_4$ such that
\[
\#\{\spanZ(\Lambda_{k} \cap B_{2^{k}}) : k \in \Z\} \leq C_3 \#\{\tilde \Lambda_k : k \in \Z\} \leq C_4.
\]
Now, for each $k$, (since the constant $C_1$ divides the constant appearing in \cref{lem:additive_to_harmonious}) the set $C_1 \cdot \spanZ(\mathscr{B}(\frac{1}{C_1} (\Lambda_k \cap B_{2^k})) \subseteq \log\mathscr{H}_+(H_k)$ is an additive, bracket-closed subgroup of $\mathfrak{g}$ whose exponential is a subgroup of $G$ that contains a generating set for $H_k$, so that
\[
\Lambda_k \subseteq C_1 \cdot \spanZ\left(\mathscr{B}\left(\frac{1}{C_1} (\Lambda_k \cap B_{2^k})\right)\right) \subseteq \log \mathscr{H}_+(H_k)
\]
for every $k$.
As such, if  $[a,b]\cap \Z$ is an interval such that $\operatorname{span}_\Z(\Lambda_k \cap B_{2^k})$ does not change as $k$ varies over $[a,b]\cap \Z$ then we have that
\[
\Lambda_a \subseteq \Lambda_b \subseteq C_1 \cdot \spanZ\left(\mathscr{B}\left(\frac{1}{C_1} (\Lambda_a \cap B_{2^a})\right)\right) \subseteq  \log \mathscr{H}_+(\Lambda_a)
\]
and hence by \cref{cor:harmonious_bounded_index} that $[H_b:H_a] \leq C_5$ for some constant $C_5$. Arguing as above, this implies that there exist constants $C_6$ and $C_7$ such that
\[
\#\{H_k :k\in \Z\} \leq C_6 \#\{\spanZ(\Lambda_{k} \cap B_{2^{k}}) : k \in \Z\}  \leq C_7,
\]
completing the proof. \qedhere
\end{proof}

Our next goal is to use \cref{prop:exploring_Carnot_subgroup} to prove the special case of \cref{thm:main_generalized} in which the group is nilpotent of bounded step.

\begin{prop}[Exploring subgroups of nilpotent groups]
\label{prop:main_nilpotent}
For each $s,k\geq 1$ there exists a constant $C(s,k)$ such that the following holds.
Let $N$ be a nilpotent group of step $s$ generated by some set $S$ with $|S|\leq k$, let $H$ be a subgroup of $N$, and for each $n\geq 1$ let $H_n$ be the subgroup of $H$ generated by elements that have word length at most $2^n$ in $(G,S)$. Then 
\[
\#\{n:H_{n+1}\neq H_n\} \leq C.
\]
\end{prop}

\begin{proof}[Proof of \cref{prop:main_nilpotent}]
We first argue that it suffices to consider the case that $N$ is equal to the free step-$s$ nilpotent group $N_{s,S}$.
Let $N$, $S$, and $(H_n)_{n\geq 0}$ be as in the statement of the theorem. Let $N_{s,S}$ be the free step-$s$ nilpotent group over $S$ and let $G_{s,S}$ be the free step-$s$ nilpotent Lie group over $S$, so that $N_{s,S}$ can be identified with the subgroup of $G_{s,S}$ generated by $S$. By the universal property of $N_{s,S}$, there exists a homomorphism $\pi:N_{s,S}\to N$ satisfying $\pi(x)=x$ for every $x\in S$, which is necessarily unique and surjective since $S$ generates $N$. Thus, $\pi$ maps the word metric $r$-ball in $(N_{s,S},S)$ to the word metric $r$-ball in $(N,S)$ for every $r\geq 0$.
For each $n \geq 0$, let $\tilde H_n$ be the subgroup of $N_{s,S}$ generated by the elements of $\pi^{-1}(H)$ that have word length at most $2^n$ in $(N_{s,S},S)$. Letting $K$ denote the kernel of $\pi$, we observe that the subgroup $K\tilde H_n=\{kh:k\in K, h\in \tilde H_n\}$ of $N_{s,S}$ is equal to the preimage $\pi^{-1}(H_n)$: On the one hand, since $K$ is normal in $K \tilde H_n$, $\pi(K \tilde H_n)$ is a subgroup of $N$ that contains the set of words in $H$ that have word length at most $2^n$, and therefore contains $H_n$. On the other hand, if $x=kh$ is an element of then we can write $h=h_1 h_2\cdots h_\ell$ as a product of elements of $\pi^{-1}(H)$ of word length at most $2^n$ in $(N_{s,S},S)$, so that $\pi(h_i)$ is an element of $H$ of word length at most $2^n$ in $(N,S)$ and $\pi(x)$ belongs to $\pi^{-1}(H_n)$ as required. 
%
%
 Since $\pi$ is surjective, we have the chain of implications
 \[
(\tilde H_{n+1}=\tilde H_n) \Rightarrow (K \tilde H_{n+1}=K\tilde H_n) \Rightarrow (H_{n+1}=H_n).
 \]
 Thus, it suffices to prove that the theorem holds with $N$ and $H$ replaced by $N_{s,S}$ and $\pi^{-1}(H)$.

\medskip

From now on we assume that $N=N_{s,S}$ and let $G=G_{s,S}$. Since $N$, $G$, and the embedding $N\to G$ are determined up to isomorphism by $s$ and $|S|$, we are now free to use constants that depend on this data in an arbitrary way (but must still be independent of the choice of subgroup $H \subseteq N$). 
By Pansu's theorem as formulated in \eqref{eq:Pansu}, there exists a left-invariant Carnot metric $d_G$ on $G$ such that 
\[
\frac{d_G(\mathrm{id},x)}{d_S(\mathrm{id},x)} \to 1 \qquad \text{ as $x\to\infty$ in $N$.}
\]
(For the argument to work we need only that the embedding of $(N,d_S)$ into $(G,d_G)$ is a quasi-isometry.)
In particular, there exist positive constants $c$ and $C$ such that
\[
\log N \cap B_{cn} \subseteq \log (\overline{S}^n) \subseteq \log N \cap B_{Cn}
\]
for every $n\geq 0$, where we write $\overline{S}=S\cup\{\mathrm{id}\}\cup S^{-1}$. This implies that if $H$ is a subgroup of $N$ then
\[
\langle H \cap B_{c2^k} \rangle \subseteq H_k \subseteq \langle H \cap B_{C2^k} \rangle
\]
for every $k \geq 0$, and the claim follows easily from this together with \cref{prop:exploring_Carnot_subgroup}. \qedhere

\end{proof}

We next deduce \cref{thm:main_generalized} from \cref{prop:main_nilpotent} and the Breuillard-Green-Tao theorem.
The proof will use the following elementary lemmas, the first of which is proven in \cite[Lemma 4.2]{MR3439705}. Recall that we write $\overline{S}=S\cup\{\mathrm{id}\}\cup S^{-1}$.

\begin{lemma}
\label{lem:generating_finite_index}
Let $G$ be a group, and let $H$ be a subgroup of $G$ with index at most $n$. If $S$ is a finite generating set for $G$, then $(\overline{S})^{2n-1} \cap H$ is a generating set for $H$.
\end{lemma}

\begin{lemma}
\label{lem:finite_index_subgroup_chain}
Let $G$ be a group, and let $H_1 \subseteq H_2 \subseteq \cdots \subseteq H_n$ and $H_1' \subseteq H_2' \subseteq \cdots \subseteq H_n'$ be two chains of subgroups such that $H'_i$ is a subgroup of $H_i$ for each $1\leq i \leq n$. Then
\[
\#\{H_i:1\leq i \leq n\} \leq (1+\lfloor \log_2 (\max_i [H_i:H_i'])\rfloor) \cdot \#\{H_{i}' : 1 \leq i \leq n\}.
\]
\end{lemma}

\begin{proof}[Proof of \cref{lem:finite_index_subgroup_chain}]
By taking a subsequence if necessary, we may assume that $H_{i+1}\neq H_i$ for every $1\leq i < n$. 
Let $\ell=\#\{H_i' : 1 \leq i \leq n \}$ and for each $0\leq k<\ell$ let $i_k$ be the $k$th time $H_i'$ changes, so that $i_0=1$ and $i_k=\min \{i>i_k : H_i' \neq H_{i_k}'\}$ for each $1 \leq k < \ell$. We also set $i_\ell=n+1$ for notational convenience. Since $H_{i_{k-1}}'=H_{i_k-1}'$ is a subgroup of $H_{i_{k-1}}$ for each $1\leq k \leq \ell$, we have that
\[
[H_{i_k-1}:H_{i_{k-1}}] \leq [H_{i_k-1}:H_{i_k-1}'] \leq \max_i [H_i:H_i']
\]
for every $1\leq k \leq \ell$.
On the other hand, we also have that
$[H_{i_k-1}:H_{i_{k-1}}] = \prod_{i=i_{k-1}}^{i_k-2} [H_{i+1}:H_i] \geq 2^{i_k-1-i_{k-1}}$ for every $1\leq k \leq \ell$,
and hence that
\[
n = \sum_{k=1}^\ell (i_k-i_{k-1}) \leq \ell \cdot (1+\lfloor \log_2 (\max_i [H_i:H_i']) \rfloor)
\]
as claimed.
\end{proof}

\begin{corollary}
\label{cor:finite_index_subgroup_chain}
Let $G$ be a group, let $G'$ be a finite-index subgroup of $G$, and let $H_1 \subseteq H_2 \subseteq \cdots \subseteq H_n$ be an increasing sequence of subgroups of $G$. Then
\[
\#\{H_i:1\leq i \leq n\}  \leq (1+\lfloor \log_2[G:G'] \rfloor) \cdot \#\{H_i \cap G':1\leq i \leq n\} .
\]
\end{corollary}

\begin{proof}
Apply \cref{lem:finite_index_subgroup_chain} with $H_i'=H_i \cap G'$ and use that $[H_i:H_i \cap G']\leq [G:G']$.
\end{proof}

We now have everything we need to prove \cref{thm:main_generalized}.

\begin{proof}[Proof of \cref{thm:main_generalized}]
Let $K\geq 1$ and let $r_0=r_0(K)$ and $C_1=C_1(K)$ be as in \cref{thm:BGT_metric}. Suppose that $r\geq r_0$ is such that $\Gr(3r)\leq K \Gr(r)$, and let $Q \triangleleft G$ and $N \triangleleft G/Q$ be as in \cref{thm:BGT_metric}.
Since $N$ has index at most $C_1$ in $G/Q$, \cref{lem:generating_finite_index} implies that the set $S':=\pi((\overline{S})^{2C_1-1}) \cap N$ is a generating set for $N$, and the word metric associated to the pair $(N,S')$ is bi-Lipschitz equivalent to the restriction of the word metric on $(G/Q,\pi(S))$ to $N$, with constants depending only on $K$.

\medskip

Let $H'=(QH)/Q$, so that $H'$ is a subgroup of $G/Q$ and $\tilde H := H' \cap N$ is a subgroup of $N$. For each $n\geq 0$ let $W'_n$ and $\tilde W_n$ be the set of the elements of $H'$ and $\tilde H$, respectively, that have word length at most $2^n$ with respect to $(G/Q,\pi(S))$, and define $H'_n= \langle W'_n \rangle$ and $\tilde H_n=\langle \tilde W_n \rangle$ for every $n\geq 0$. Since $[H'_n:H'_n \cap N]\leq [G/Q:N]\leq C_1$ and $W'_n$ is a finite symmetric generating set for $H'_n$, the set $(W'_n)^{2C_1-1}\cap N$ is a generating set for $H'_n \cap N$, so that if we define $m=\lceil \log_2 (2C_1-1) \rceil$ then
\begin{equation}
\label{eq:H'H_containment}
H'_n \cap N \subseteq \tilde H_{n+m} \subseteq H'_{n+m} \cap N
\end{equation}
for every $n\geq 0$.

\medskip

   Using the above mentioned bi-Lipschitz equivalence between the two different word metrics on $N$ and the fact that the step of $N$ and the size of the generating set $S'$ are bounded by constants depending only on $K$ and $k$, it follows from \cref{prop:main_nilpotent} that there exists a constant $C_2=C_2(K,k)$ such that
$
\#\{\tilde H_n : n \geq 0 \} \leq C_2.
$
It follows from this and \eqref{eq:H'H_containment} that there 
exists a constant $C_3=C_3(K,k)$ such that
$
\#\{N\cap H_n' : n \geq 0 \} \leq C_3,
$
and hence by \cref{cor:finite_index_subgroup_chain} that there exists a constant $C_4=C_4(K,k)$ such that
$
\#\{ H_n' : n \geq 0 \} \leq C_4.
$

\medskip

Now, as in the proof of \cref{prop:main_nilpotent}, we have that $QH_n = \pi^{-1}(H'_n)$ for every $n\geq 0$.
Observe that if $QH_{n+1} = QH_n$ but $H_{n+1}\neq H_n$ then there exist $q_n,q_{n+1}\in Q$, $h_n \in H_n$, and $h_{n+1}\in H_{n+1}\setminus  H_n$ such that $q_nh_n=q_{n+1} h_{n+1}$, and hence that $h_{n+1}h_n^{-1}=q_{n+1}^{-1}q_n \in Q$. Since $Q$ has diameter at most $C_1r$, this implies that $h_{n+1}h_n^{-1}$ has word length at most $C_1r$, contradicting the assumption that $h_{n+1}\notin H_n$ if $2^n \geq C_1 r$. 
It follows that there exists a constant $C_5=C_5(K)$ such that
\[
\#\{ n \geq C_5 + \log_2 r:H_{n+1} \neq H_n \} \leq \#\{ n \geq C_5 + \log_2 r: Q H_{n+1} \neq QH_n \} \leq \#\{ H_n' : n \geq 0 \} \leq C_4,
\]
which easily implies the claim.
\end{proof}

It remains only to deduce \cref{thm:main} from \cref{thm:main_generalized}. We will need the following lemma about the injectivity radius of the quotient $F_S / \llangle R_n\rrangle \to G$. For each $n\geq 0$ let $G_n = F_S / \llangle R_n \rrangle$.

\begin{lem} \label{lem:small_presentation_means_balls_look_similar}
      Let $G$ be a group with a finite generating set $S$.
       The quotient map $G_{n}\to G$ induces a map between Cayley graphs that restricts to an isomorphism between the balls of radius $2^{n-1}-1$ around the identity.
\end{lem}

\begin{proof}[Proof of \cref{lem:small_presentation_means_balls_look_similar}]
      %
      It suffices to prove that the quotient 
      map $\pi:G_{n} \to G$ is injective on $(\overline{S})^{2^{n-1}}$. (This implies that $\pi(x s)=\pi(y)$ if and only if $xs=y$ for every $s\in \overline{S}$ and $x,y$ in the ball of radius $2^{n-1}-1$ and hence that the balls of radius $2^{n-1}-1$ are isomorphic.)
      Suppose for contradiction that this is false. Then there exist $u,v \in (\overline{S})^{2^{n-1}} \subseteq F_S$ such that
      $u^{-1} v \in R \setminus \llangle R_{n} \rrangle$.
Since $u$ and $v$ both belong to $(\overline{S})^{2^{n-1}}$, the product   $u^{-1}v$ belongs to $(\overline{S})^{2^{n}}$, and since it also belongs to $R$ it must belong to $R_{n}$ by definition of $R_{n}$. This contradicts the assumption that $u^{-1}v \notin \llangle R_{n} \rrangle$.
\end{proof}

\begin{proof}[Proof of \cref{thm:main}]
Let $r_0=r_0(K)$ and $C=C(K,k)$ be the constants from \cref{thm:main_generalized}.
Let $n_0=\lceil 4+\log_2 r \rceil$ and let $G'=F_S/\llangle R_{n_0} \rrangle$. The projection $G' \to G$ induces a surjective graph homomorphism between the Cayley graphs $\Cay(G',S)$ and $\Cay(G,S)$ that restricts to an isomorphism between the balls of radius $2^{n_0-1}-1\geq 4r$. Let $H$ be the subgroup of $G'$ generated by $\bigcup_{n\geq n_0} (R_n/\llangle R_{n_0} \rrangle)$ and, for each $n\geq n_0$, let $H_n$ be the subgroup of $H$ generated by $R_n/\llangle R_0 \rrangle$. If $r\geq r_0$, we may apply \cref{thm:main_generalized} to $G'$ and $H$ to obtain that
$\#\{H_n : n \geq n_0\} \leq C$, and it follows that 
\[
\#\{\llangle R_n \rrangle : n \geq n_0\}=\#\{\llangle H_n \rrangle : n \geq n_0\} \leq \#\{H_n : n \geq n_0\} \leq C
\]
as claimed.
\end{proof}





\section*{Acknowledgments} 
TH thanks Christian Gorski and Mikolaj Fraczyk for helpful conversations and thanks Matthew Tointon and 
Seung-Yeon Ryoo for comments on a draft. This work was supported by NSF grant DMS-2246494. Both authors thank the anonymous referees for their comments. 

\bibliographystyle{amsplain}

\begin{thebibliography}{99}
\bibitem{MR2773031}
I.~Benjamini, A.~Nachmias, and Y.~Peres.
\newblock Is the critical percolation probability local?
\newblock {\em Probab. Theory Related Fields}, 149(1-2):261--269, 2011.

\bibitem{tessera2023smalldoublingimpliessmall}Tessera, R. \& Tointon, M. Small doubling implies small tripling at large scales.  (2023), https://arxiv.org/abs/2310.20500



\bibitem{berestycki2022universality}
N.~Berestycki, J.~Hermon, and L.~Teyssier.
\newblock On the universality of fluctuations for the cover time.
\newblock {\em arXiv preprint arXiv:2202.02255}, 2022.

\bibitem{bonfiglioli2007stratified}
A.~Bonfiglioli, E.~Lanconelli, and F.~Uguzzoni.
\newblock {\em Stratified Lie groups and potential theory for their
  sub-Laplacians}.
\newblock Springer Science \& Business Media, 2007.

\bibitem{breuillard2014brief}
E.~Breuillard.
\newblock A brief introduction to approximate groups.
\newblock {\em Thin groups and superstrong approximation}, 61:23--50, 2014.

\bibitem{breuillard2011approximate}
E.~Breuillard and B.~Green.
\newblock Approximate groups. {I} the torsion-free nilpotent case.
\newblock {\em Journal of the Institute of Mathematics of Jussieu},
  10(1):37--57, 2011.

\bibitem{MR2827010}
E.~Breuillard, B.~Green, and T.~Tao.
\newblock Approximate subgroups of linear groups.
\newblock {\em Geom. Funct. Anal.}, 21(4):774--819, 2011.

\bibitem{bgt12}
E.~Breuillard, B.~Green, and T.~Tao.
\newblock The structure of approximate groups.
\newblock {\em Publ. Math. Inst. Hautes Études Sci.}, 116:115--221, 2012.

\bibitem{breuillard2013rate}
E.~Breuillard and E.~Le~Donne.
\newblock On the rate of convergence to the asymptotic cone for nilpotent
  groups and subfinsler geometry.
\newblock {\em Proceedings of the National Academy of Sciences},
  110(48):19220--19226, 2013.

\bibitem{MR3439705}
E.~Breuillard and M.~C.~H. Tointon.
\newblock Nilprogressions and groups with moderate growth.
\newblock {\em Adv. Math.}, 289:1008--1055, 2016.

\bibitem{bridson2002geometry}
M.~R. Bridson.
\newblock The geometry of the word problem.
\newblock {\em Invitations to geometry and topology}, 7:29--91, 2002.

\bibitem{casas2012efficient}
F.~Casas, A.~Murua, and M.~Nadinic.
\newblock Efficient computation of the {Z}assenhaus formula.
\newblock {\em Computer Physics Communications}, 183(11):2386--2391, 2012.

\bibitem{MR1434478}
J.~W.~S. Cassels.
\newblock {\em An introduction to the geometry of numbers}.
\newblock Classics in Mathematics. Springer-Verlag, Berlin, 1997.
\newblock Corrected reprint of the 1971 edition.

\bibitem{contreras2021supercritical}
D.~Contreras, S.~Martineau, and V.~Tassion.
\newblock Supercritical percolation on graphs of polynomial growth.
\newblock {\em arXiv preprint arXiv:2107.06326}, 2021.

\bibitem{contreras2023locality}
D.~Contreras, S.~Martineau, and V.~Tassion.
\newblock Locality of percolation for graphs with polynomial growth.
\newblock {\em Electronic Communications in Probability}, 28:1--9, 2023.

\bibitem{MR1232845}
T.~Coulhon and L.~Saloff-Coste.
\newblock Isop\'{e}rim\'{e}trie pour les groupes et les vari\'{e}t\'{e}s.
\newblock {\em Rev. Mat. Iberoamericana}, 9(2):293--314, 1993.

\bibitem{1806.07733}
H.~Duminil-Copin, S.~Goswami, A.~Raoufi, F.~Severo, and A.~Yadin.
\newblock Existence of phase transition for percolation using the gaussian free
  field.
\newblock {\em Duke Mathematical Journal}, 169(18):3539--3563, 2020.



\bibitem{fisher2009freiman}
D.~Fisher, N.~H. Katz, and I.~Peng.
\newblock On {F}reiman's theorem in nilpotent groups.
\newblock {\em arXiv preprint arXiv:0901.1409}, 2009.

\bibitem{gersten2003isoperimetric}
S.~M. Gersten, D.~F. Holt, and T.~R. Riley.
\newblock Isoperimetric inequalities for nilpotent groups.
\newblock {\em Geometric \& Functional Analysis GAFA}, 13:795--814, 2003.

\bibitem{MR49203}
A.~M. Gleason.
\newblock Groups without small subgroups.
\newblock {\em Ann. of Math. (2)}, 56:193--212, 1952.

\bibitem{gromov81poly}
M.~Gromov.
\newblock Groups of polynomial growth and expanding maps.
\newblock {\em Publ. Math. Inst. Hautes Études Sci.}, 53:53--73, 1981.

\bibitem{MR2833482}
E.~Hrushovski.
\newblock Stable group theory and approximate subgroups.
\newblock {\em J. Amer. Math. Soc.}, 25(1):189--243, 2012.

\bibitem{hutchcroft2021non}
T.~Hutchcroft and M.~Tointon.
\newblock Non-triviality of the phase transition for percolation on finite
  transitive graphs.
\newblock {\em arXiv preprint arXiv:2104.05607}, 2021.

\bibitem{MR2629989}
B.~Kleiner.
\newblock A new proof of {G}romov's theorem on groups of polynomial growth.
\newblock {\em J. Amer. Math. Soc.}, 23(3):815--829, 2010.

\bibitem{lazard1954groupes}
M.~Lazard.
\newblock Sur les groupes nilpotents et les anneaux de {L}ie.
\newblock In {\em Annales scientifiques de l'{\'E}cole Normale Sup{\'e}rieure},
  volume~71, pages 101--190, 1954.

\bibitem{lyons2023explicit}
R.~Lyons, A.~Mann, R.~Tessera, and M.~Tointon.
\newblock Explicit universal minimal constants for polynomial growth of groups.
\newblock {\em Journal of Group Theory}, 26(1):29--53, 2023.

\bibitem{MR49204}
D.~Montgomery and L.~Zippin.
\newblock Small subgroups of finite-dimensional groups.
\newblock {\em Ann. of Math. (2)}, 56:213--241, 1952.

\bibitem{ozawa2018functional}
N.~Ozawa.
\newblock A functional analysis proof of {G}romov's polynomial growth theorem.
\newblock {\em Annales scientifiques de l'{\'E}cole normale sup{\'e}rieure},
  51(3):549--556, 2018.

\bibitem{panagiotis2021gap}
C.~Panagiotis and F.~Severo.
\newblock Gap at 1 for the percolation threshold of {C}ayley graphs.
\newblock {\em arXiv preprint arXiv:2111.00555}, 2021.

\bibitem{pansu1983croissance}
P.~Pansu.
\newblock Croissance des boules et des g{\'e}od{\'e}siques ferm{\'e}es dans les
  nilvari{\'e}t{\'e}s.
\newblock {\em Ergodic Theory and Dynamical Systems}, 3(3):415--445, 1983.

\bibitem{raghunathan1972discrete}
M.~S. Raghunathan.
\newblock {\em Discrete subgroups of Lie groups}, volume~3.
\newblock Springer, 1972.

\bibitem{shalom-tao}
Y.~Shalom and T.~Tao.
\newblock A finitary version of {G}romov's polynomial growth theorem.
\newblock {\em Geom. Funct. Anal.}, 20:1502--1547, 2010.

\bibitem{MR3237440}
T.~Tao.
\newblock {\em Hilbert's fifth problem and related topics}, volume 153 of {\em
  Graduate Studies in Mathematics}.
\newblock American Mathematical Society, Providence, RI, 2014.

\bibitem{tao2017inverse}
T.~Tao.
\newblock Inverse theorems for sets and measures of polynomial growth.
\newblock {\em The Quarterly Journal of Mathematics}, 68(1):13--57, 2017.

\bibitem{tashiro2022speed}
K.~Tashiro.
\newblock On the speed of convergence to the asymptotic cone for non-singular
  nilpotent groups.
\newblock {\em Geometriae Dedicata}, 216(2):18, 2022.

\bibitem{tessera2017scaling}
R.~Tessera and M.~Tointon.
\newblock Scaling limits of {C}ayley graphs with polynomially growing balls.
\newblock {\em arXiv preprint arXiv:1711.08295}, 2017.

\bibitem{MR3877012}
R.~Tessera and M.~Tointon.
\newblock Properness of nilprogressions and the persistence of polynomial
  growth of given degree.
\newblock {\em Discrete Anal.}, pages Paper No. 17, 38, 2018.

\bibitem{tt.Trof}
R.~Tessera and M.~Tointon.
\newblock A finitary structure theorem for vertex-transitive graphs of
  polynomial growth.
\newblock {\em Combinatorica}, 2019.
\newblock To appear.

\bibitem{tt.resist}
R.~Tessera and M.~Tointon.
\newblock Sharp relations between volume growth, isoperimetry and resistance in
  vertex-transitive graphs.
\newblock {\em arXiv preprint arXiv:2001.01467}, 2020.

\bibitem{tointon2020brief}
M.~C. Tointon.
\newblock A brief introduction to approximate groups.
\newblock {\em European Mathematical Society Magazine}, (115):12--16, 2020.

\bibitem{MR3971253}
M.~C.~H. Tointon.
\newblock {\em Introduction to approximate groups}, volume~94 of {\em London
  Mathematical Society Student Texts}.
\newblock Cambridge University Press, Cambridge, 2020.

\bibitem{MR811571}
V.~I. Trofimov.
\newblock Groups of automorphisms of graphs as topological groups.
\newblock {\em Mat. Zametki}, 38(3):378--385, 476, 1985.

\bibitem{MR832044}
N.~T. Varopoulos.
\newblock Th\'{e}orie du potentiel sur des groupes et des vari\'{e}t\'{e}s.
\newblock {\em C. R. Acad. Sci. Paris S\'{e}r. I Math.}, 302(6):203--205, 1986.

\bibitem{Locality}Easo, P. \& Hutchcroft, T. The critical percolation probability is local.  (2023), https://arxiv.org/abs/2310.10983


\bibitem{Wang2023}
Z.~Wang.
\newblock {MATH} 597 : {D}ynamical {S}ystems on {N}ilmanifolds.
\newblock Unpublished lecture notes. Available at:
  \url{https://www.personal.psu.edu/zxw14/MATH597/index.html}, 2023.

\bibitem{MR36766}
H.~Yamabe.
\newblock On an arcwise connected subgroup of a {L}ie group.
\newblock {\em Osaka Math. J.}, 2:13--14, 1950.

\bibitem{tessera2024ballsgroupsvolumestructure}Tessera, R. \& Tointon, M. Balls in groups: volume, structure and growth.  (2024), https://arxiv.org/abs/2403.02485


\end{thebibliography}


\begin{dajauthors}
\begin{authorinfo}[philip]
  Philip Easo\\
  The Division of Physics, Mathematics and Astronomy\\
  California Institute of Technology\\
  Pasadena, California, USA\\
  peaso\imageat{}caltech\imagedot{}edu \\
  \url{https://philipeaso.com/}
\end{authorinfo}
\begin{authorinfo}[johan]
  Tom Hutchcroft\\
  The Division of Physics, Mathematics and Astronomy\\
  California Institute of Technology\\
  Pasadena, California, USA\\
  t.hutchcroft\imageat{}caltech\imagedot{}edu \\
  \url{https://www.its.caltech.edu/~thutch/}
\end{authorinfo}
\end{dajauthors}

\end{document}